\RequirePackage{fix-cm}
\documentclass[smallcondensed, final, numbook]{svjour3}     
\smartqed  
 \usepackage{mathptmx}      

\usepackage{latexsym}
\usepackage{lipsum}
\usepackage{amsfonts}
\usepackage{amsmath}
\usepackage{graphicx}
\usepackage{algorithm}
\usepackage{algorithmic}
\usepackage{verbatim}
\usepackage{pgfplots}
\usepackage[caption=false]{subfig}
\usepackage{multirow}
\usepackage{amsopn}
\usepackage[misc]{ifsym}
\usepackage{hyperref}
\usepackage{cleveref}
\usepackage{microtype}
\Crefname{ALC@unique}{Line}{Lines}
\DeclareMathOperator{\diag}{diag}
 \journalname{myjournal}
\begin{document}

\title{A FEAST SVDsolver based on Chebyshev--Jackson series for computing partial singular triplets of large matrices\thanks{Supported in part by the National Natural Science Foundation of China (No. 12171273)}}


\titlerunning{A CJ--FEAST SVDsolver}        

\author{ Zhongxiao Jia\textsuperscript{1}   \and Kailiang Zhang\textsuperscript{1}}

\authorrunning{Z. JIA AND K. Zhang} 

\institute{
           \begin{itemize}
            \item[\Letter] {Zhongxiao Jia} \\
                jiazx@tsinghua.edu.cn \\
            \item[] Kailiang Zhang \\
                zkl18@mails.tsinghua.edu.cn \\
                \at
            \item[\textsuperscript{1}] Department of Mathematical Sciences, Tsinghua University, 100084 Beijing, China
          \end{itemize}
}
\date{Received: date / Accepted: date}

\maketitle

\begin{abstract}
    The FEAST eigensolver is extended to the computation of the singular
    triplets of a large matrix $A$ with the singular values in a given interval.
    The resulting FEAST SVDsolver is subspace iteration
    applied to an approximate spectral projector of $A^TA$ corresponding to
    the desired singular values in a given interval, and constructs approximate
    left and right singular subspaces corresponding to the desired singular
    values, onto which $A$ is projected to obtain Ritz approximations. Differently from a commonly used contour
    integral-based FEAST solver, we propose a robust alternative
    that constructs approximate spectral projectors by
    using the Chebyshev--Jackson polynomial series, which are
    symmetric positive semi-definite with the eigenvalues in $[0,1]$.
    We prove the pointwise convergence of this series
    and give compact estimates for pointwise errors of
    it and the step function that corresponds to the exact spectral
    projector. We present error bounds for the approximate spectral projector
    and reliable estimates for the number of desired singular triplets,
    establish numerous convergence results on the resulting FEAST SVDsolver, and
    propose practical selection strategies for determining
    the series degree and for reliably determining the subspace dimension. The solver and
    results on it are directly applicable or adaptable to the real symmetric and complex
    Hermitian eigenvalue problem.
    Numerical experiments illustrate that our FEAST SVDsolver is at least
    competitive with and is much more efficient than the
    contour integral-based FEAST SVDsolver when the desired singular values are
    extreme and interior ones, respectively, and it
    is also more robust than the latter.
\end{abstract}

\section{Introduction}

Matrix singular value decomposition (SVD) problems play a crucial role in many
applications. For small to moderate problems,
very efficient and robust SVD algorithms and softwares have been well
developed and widely used \cite{golub2013matrix,stewart2001matrix}.
They are often called direct SVD solvers, and compute the entire singular
values and/or singular vectors using predictable iterations.
In this paper, we consider the following partial SVD problem:
Given a matrix $A\in\mathbb{R}^{m\times n}$ with $m\geq n\gg 1$ and a real
interval $[a,b]$ contained in the singular spectrum of $A$,
determine the $n_{sv}$ singular triplets $(\sigma,u,v)$ with the singular
values $\sigma\in [a,b]$ counting multiplicities,
where
\begin{equation*}
    \begin{cases}
        Av=\sigma u,   \\
        A^Tu=\sigma v, \\
        \left\|u\right\|=\left\|v\right\|=1.
    \end{cases}
\end{equation*}

Since the SVD of $A$ is mathematically equivalent to the
eigendecomposition of its cross-product matrix $A^TA$,
it is possible to adapt those algorithms for
a symmetric matrix eigenvalue problem to the corresponding
SVD problem in some numerically stable way.
Over the past two decades,
a new class of numerical methods has emerged for computing the eigenvalues
of a large matrix in a given region and/or the associated eigenvectors,
and they are based on contour integration and rational filtering.
Among them, representatives are the Sakurai--Sugiura (SS) method \cite{sakurai2003projection} and the FEAST eigensolver \cite{polizzi2009density},
which fall into the category of Rayleigh--Ritz projection methods. We should
point out that, for the computation of
eigenvalues in a given region inside the spectrum, all the other available
algorithms, e.g., subspace iteration, Arnoldi type algorithms and their
shift-invert variants, and Jacobi--Davidson type algorithms, are not
directly applicable. The only exception is that the given
region and exterior eigenvalues coincide and the number
of eigenvalues in the region is known. In this case, the implicitly restarted Arnoldi
algorithm \cite{sorensen1992}, on which the package ARPACK \cite{lehoucq1998}
and the Matlab function {\textsf eigs} are based, and the implicitly restarted
refined Arnoldi algorithm \cite{jia1999} can be used.

The SS method and subsequent variants
\cite{ikegami2010contour,ikegami2010filter,imakura2014SSarnoldi,sakurai2016,sakurai2007cirr}
have resulted in the z-Pares package \cite{Futamura2014Online} that handles
large Hermitian and non-Hermitian matrix eigenvalue problems.
The original SS method is the SS-Hankel method,
and its variants include the SS-RR method (Rayleigh--Ritz projection) and
the SS-Arnoldi method as well as their block variants.
The SS-Hankel method computes certain moments,
which are constructed by the contour integrals with an integral domain containing all the desired eigenvalues,
to form small Hankel matrices or matrix pairs of order $m$,
whose eigenvalues equal the desired $m$ distinct eigenvalues of the original matrix or matrix pair contained in the region.
In computations,
one computes those contour integrals by some numerical quadrature and obtains approximations to the moments,
or constructs an approximate spectral
projector associated with all the desired eigenvalues if the exact
spectral projector is involved.
The SS method and its variants are essentially Krylov or block Krylov subspace
based methods starting with a {\em specific initial vector or block vector
that is generated by
acting the approximate spectral projector on a vector or block vector chosen
randomly}, realize the Rayleigh--Ritz projection onto them,
and compute Ritz approximations \cite{sakurai2016}.
The SS-RR method computes an orthonormal basis of the underlying subspace
and projects the large matrix or matrix pair onto it,
and the SS-Arnoldi method exploits the Arnoldi process to generate
an orthonormal basis of the subspace and forms the projection matrix.
We refer the reader to \cite{sakurai2016} for a summary of these methods.

The FEAST eigensolver \cite{guttel2015zolotarev,kestyn2016feast,polizzi2009density,tang2014feast},
first introduced by Polizzi \cite{polizzi2009density} in 2009, has led to
the development of the FEAST numerical library \cite{polizzi2020feast}. Unlike the SS
method and its variants, this eigensolver
works on subspaces of a fixed dimension and
uses subspace iteration \cite{golub2013matrix,parlett1998symmetric,saad2011numerical,stewart2001matrix}
on an approximate spectral projector to generate a sequence of subspaces, onto which
the Rayleigh--Ritz projection of the original matrix or matrix pair is realized
and the Ritz approximations are computed.

In the SS-method and the FEAST eigensolver, since
the spectral projector associated with the eigenvalues in a given
region can be represented in the form of a contour integral, computationally
they use a suitably chosen quadrature to approximate
the integral and construct an approximate
spectral projector. 
This involves solutions of several linear systems with shifted coefficient
matrices, where the shifts are the quadrature nodes. For instance,
the FEAST eigensolver needs to solve several,
i.e., the subspace dimension times the number of nodes, large linear systems
at each iteration.
If the matrix is structured, such as banded, then one can use
LU factorizations \cite{golub2013matrix} to solve the linear systems
involved efficiently. But if the matrix is generally dense or sparse,
ones needs to apply some iterative methods, e.g.,
Krylov subspace iterative methods, to solve
them approximately, and the resulting algorithm is called IFEAST
\cite{gavin2018ifeast}. However, these linear systems are highly indefinite
when the region of interest is
inside the spectrum. It is well known that, for highly indefinite or
nonsymmetric linear systems, Krylov subspace iterative solvers,
e.g., the GMRES and BiCGstab methods \cite{saad2003},
are generally inefficient and can be very slow.
An adaptation of Theorem 3.1 of \cite{robbe2009} on inverse
subspace iteration to the current context states
that these shifted linear systems must be solved with {\em increasing
accuracy} in order to guarantee that the FEAST eigensolver converges
linearly \cite{gavin2018ifeast}. As a consequence,
the FEAST eigensolver may be extremely slow even if these linear systems
are solved in parallel. We should point out
that there has not yet been a general effective preconditioning technique for
highly indefinite linear systems.

As a matter of fact, the situation is more subtle. It is known
from \cite{tang2014feast} that the distance between a desired eigenvector
and the subspace may only decrease down to the relative accuracy level of the
approximate solutions of the shifted linear systems {\em rather than}
the residual norm level. This implies that, in finite precision,
the residual norm of an approximate eigenpair by the FEAST solver may not
drop below a reasonably prescribed tolerance, say $10^{-13}$, once one of the shifted
linear systems is ill conditioned, which is definitely true when some of the
nodes are close to some eigenvalues of the underlying matrix.

More precisely, it is well known that
the {\em attainable} relative error, i.e., the relative accuracy, of an
approximate solution is bounded by
the condition number times by the relative residual norm. This
error bound is in the worst case but is achievable.
Suppose that the condition number of a shifted linear system is no less
than $10^3$ and the relative error bound for the approximate solution
is attainable.
Then, in finite precision, even if the relative residual norm is already
as small as $10^{-15}\sim10^{-14}$, i.e., the level of machine precision $2.22\times 10^{-16}$,
the relative accuracy of the approximate solution may only
achieve $10^3\times 10^{-14}=10^{-11}$.
As a result,
the attainable relative residual norms of approximate eigenpairs
by the contour integral-based FEAST eigensolver may not
decrease to $10^{-13}$, meaning that it
fails to converge for a prescribed reasonable
stopping tolerance $10^{-13}$.
As is pointed out in \cite{kestyn2016feast}, such a case
occurs more possibly for the non-Hermitian matrix eigenvalue problem and
could also occur in the Hermitian case. In principle,
a possible remedy is to take nodes away from the real axis, but how to treat
it effectively is nontrivial, and there is no systematic and viable solution.
In computations, whenever this case occurs,
there may be two consequences.
First, the FEAST eigensolver itself
may not converge, as Theorem 4.4 of \cite{tang2014feast} indicates,
because the convergence conditions there may not be met. Second,
although it converges, the
distance between a desired eigenvector and the subspace may decrease
only to the level of the accuracy of the approximate solutions of
shifted linear systems, as described above.
Therefore,
on the one hand, it may be very costly to solve them; on the other
hand, approximate solutions may not achieve the desired accuracy requirement,
causing that the FEAST eigensolver may have robustness problem if
higher but reasonable accuracy is required in finite precision. We will
present an example to illustrate
this in the section of later numerical experiments.

In this paper, putting aside the representation of contour integral,
we notice that the underlying spectral projector precisely corresponds to
a specific step or piecewise continuous function $h(x)$, which will be defined
later. This makes it possible to propose other alternatives to
construct a good approximate spectral
projector without solutions of shifted linear systems at
each iteration and, meanwhile, to improve the overall efficiency
and robustness of this kind of solvers.
An obvious alternative is to approximate $h(x)$
by algebraic polynomials and then constructs an approximate spectral projector
correspondingly. For instance, we can do these by
the famous Chebyshev or Chebyshev--Jackson series expansion.
Such approximations are not new, and have been mentioned
and briefly considered in, e.g., \cite{di2016efficient}.
However, except \cite{di2016efficient},
such polynomial approximation approach received
little attention, compared with
rational approximations based on the contour integral and quadratures.
Among others, a fundamental cause is that it lacks
the pointwise convergence of the Chebyshev--Jackson series and
its pointwise error estimates as well as
accuracy estimates for the approximate spectral projector.

It is well known from, e.g., \cite{mason2002chebyshev} that the Chebyshev
series expansion is the best least squares approximation to a given function
with respect to
the Chebyshev $l_2$-norm. For the step function $h(x)$,
the researchers in \cite{di2016efficient} derive a quantitative error
estimate for the mean-square convergence of
Chebyshev series approximation. However, it is the
{\em pointwise} error of the series and its quantitative error estimates
that matter and are critically needed. Unfortunately, the mean-square
convergence does not necessarily mean the pointwise convergence,
and one cannot obtain desired error estimates
from those mean-square convergence results either.
For the step function $h(x)$, it is shown in,
e.g., \cite{di2016efficient} that Jackson coefficients
\cite{rivlin1981introduction} can considerably dampen Gibbs
oscillations, and it is thus better to exploit the Chebyshev--Jackson
series. However, the pointwise convergence of this series and
its quantitative error estimates also
lack for this series.
As a consequence, nothing has been known on the convergence of the
the resulting FEAST eigensolver, let alone
a reliable determination of the subspace dimension $p$
and a proper selection of the series degree $d$ when using the
Chebyshev--Jackson series to construct an approximate spectral projector
in order to propose and develop a convergent FEAST eigensolver.

The FEAST eigensolver can be directly adapted to the
computation of the singular triplets of $A$ associated with
the singular values $\sigma$ in a given interval $[a,b]$ in some numerically
stable way. Precisely, for such a partial SVD problem, we
will construct an approximate spectral projector of $A^TA$
associated with $\sigma\in [a,b]$ by
exploiting the Chebyshev--Jackson series expansion,
apply subspace iteration
to the approximate spectral projector constructed, and
generate a sequence of approximate {\em left} and {\em right} singular
subspaces corresponding to $\sigma\in [a,b]$. In computations, for
numerical stability,
instead of working on the eigenvalue problem of $A^TA$, we work on $A$ directly, project $A$ onto the left and right subspaces generated, and compute
the Ritz approximations to the desired singular triplets. We call
the resulting algorithm the Chebyshev--Jackson FEAST (CJ-FEAST) SVDsolver.

For the CJ-FEAST SVDsolver, we will make a detailed analysis
of the pointwise convergence
of the Chebyshev--Jackson series, and establish sharp
{\em pointwise} error estimates for the series.
Particularly, we prove that the values of the Chebyshev--Jackson series always
lie in $[0,1]$, which will make the approximate
spectral projectors unconditionally symmetric positive
semi-definite (SPSD) and their eigenvalues always lie in $[0,1]$.
We make full use of these results to estimate the accuracy of the approximate spectral projector and prove the convergence of the CJ-FEAST SVDsolver.
We establish the estimates for the distances of approximate
subspaces and the desired right singular subspace,
show how each of the Ritz approximations converges, and give
the convergence rates of Ritz values and left and right Ritz vectors.
Also, exploiting the pointwise convergence results
and randomized trace estimation results
\cite{avron2011randomized,Cortinovis2021onrandom,roosta2015improved}, we
give reliable estimates for the number $n_{sv}$ of desired singular triplets
with $\sigma\in [a,b]$.
These estimates are useful for all FEAST-type methods and SS-type methods.
With these results, we are able to propose practical
and robust selection strategies for determining the series degree and
for ensuring the subspace dimension $p\geq n_{sv}$.
Unlike the contour integral-based FEAST SVDsolver, the attainable accuracy,
i.e., the residual norms of approximate singular triplets obtained by
the CJ-FEAST SVDsolver can achieve the level of machine precision regardless
of the singular value distribution and without additional
requirements. Compared with the contour integral-based FEAST SVDsolver, another attractive property of the CJ-FEAST SVDsolver is that its computational cost does not depend on
whether or not the interval of interest corresponds to exterior or interior
singular values.

All the theoretical results and algorithms in this paper
are directly applicable or adaptable to the
real symmetric and complex Hermitian matrix eigenvalue problems, once
we replace $A^TA$ by a given matrix itself and the Rayleigh--Ritz projection
for the SVD problem by that for the eigenvalue problem.
We should particularly point out that, similarly to
a contour integral-based FEAST solver where the shifted linear systems
can be solved in parallel at each iteration, the action of an approximate
spectral projector on several vectors can be
realized in parallel too.

The paper is organized as follows.
In \Cref{sec: Known results}, we review some preliminaries, the subspace
iteration applied to an approximate spectral projector
and some results to be used in the paper.
In \Cref{sec: Chebyshev series convergence}, we establish compact
quantitative pointwise convergence results on the Chebyshev--Jackson series.
Then we propose the CJ-FEAST SVDsolver in \Cref{sec: cross product method}
to compute the $n_{sv}$ desired singular triplets of $A$. We establish
estimates for accuracy of the approximate spectral projector and
the number of desired singular values. In \Cref{conver}, we establish
the convergence of the CJ-FEAST SVDsolver,
and present a number of convergence results.
In \Cref{sec: experiments}, we report numerical experiments to illustrate the
performance of the CJ-FEAST SVDsolver. We also make a comparison of
our solver and the IFEAST eigensolver applied
to the SVD problem, and illustrate the competitiveness,
superiority and robustness of our solver.
Finally, we conclude the paper in \Cref{sec:end}.

Throughout this paper, denote by $\|\cdot\|$ the 2-norm of a vector or matrix,
by $I_n$ the identity matrix of order $n$ with $n$ dropped whenever it is
clear from the context, by $e_i$ column $i$ of $I_n$,
and by $\sigma_{\max}(X)$ and $\sigma_{\min}(X)$ the largest and
smallest singular values of a matrix $X$, respectively.
For the concerning SVD problem of a matrix
$A\in \mathbb{R}^{m\times n}$ with $m<n$,
we simply apply the algorithm to $A^T$.

\section{Preliminaries and a basic algorithm}\label{sec: Known results}
Denote by $S=A^TA$, and let
\begin{equation*}
    A=U\begin{pmatrix}
        \Sigma \\
        \text{\large 0}
    \end{pmatrix}V^T
\end{equation*}
be the SVD of $A$ with the diagonals $\sigma$'s of $\Sigma$ being the
singular values and the columns of $U$ and $V$ being the
corresponding left and right singular vectors; see \cite{golub2013matrix}. Then
\begin{equation}
    V^TSV=\Sigma^2\in \mathbb{R}^{n\times n} \label{ceigen}
\end{equation}
is the eigendecomposition of $S$. At this moment we do not label
the order of the singular values $\sigma$'s.

Given an interval $[a,b]\subset [\sigma_{\min},\|A\|]$ with $\sigma_{\min}=
\sigma_{\min}(A)$, suppose that
we are interested in all the singular values $\sigma\in [a,b]$ of $A$ and/or
the corresponding left and right singular vectors.
Define
\begin{equation}\label{ps1}
    P_{S}=V_{in}V_{in}^T+\frac{1}{2}V_{ab}V_{ab}^T,
\end{equation}
where $V_{in}$ consists of the columns of $V$ corresponding to the eigenvalues of
$S$ in the open interval $(a^2,b^2)$ and $V_{ab}$ consists of the columns
of $V$ corresponding to the eigenvalues of $S$ that equal the end
$a^2$ or $b^2$.
Notice that if neither of $a$ nor $b$ is a singular value of $A$
then $P_{S}=V_{in}V_{in}^T$ is the standard spectral
projector of $S$ associated with its eigenvalues $\sigma^2\in [a^2,b^2]$.
If either $a$ or $b$ or both them are singular values, then
$P_{S}$ is called a generalized spectral projector associated with
all the $\sigma\in [a,b]$. The factor $\frac{1}{2}$ is
necessary, and it corresponds to the step function to
be introduced later that is approximated by the Chebyshev--Jackson
series in this paper or by a rational function in the context
of the contour integral. In the sequel, we simply call $P_{S}$
the spectral projector of $S$ associated with $\sigma\in [a,b]$.

For an approximate singular triplet $(\hat{\sigma},\hat{u},\hat{v})$ of $A$,
its residual is
\begin{equation}\label{resnorm}
    r=r(\hat{\sigma}, \hat{u}, \hat{v}):=\begin{bmatrix}
            A\hat{v}-\hat{\sigma}\hat{u} \\
            A^T\hat{u}-\hat{\sigma}\hat{v}
        \end{bmatrix},
\end{equation}
and the size of $\|r\|$ will be used to decide the convergence of
$(\hat{\sigma},\hat{u},\hat{v})$.

\Cref{alg:subspace iteration} is an algorithmic framework of the
FEAST SVDsolver, where $P$ is an approximation to $P_{S}$. It
is the Rayleigh--Ritz projection with respect to
the left and right subspaces $\mathcal{U}^{(k)}$ and
$\mathcal{V}^{(k)}$ for the SVD problem, where
$\mathcal{U}^{(k)}=A\mathcal{V}^{(k)}$, and computes
the Ritz approximations
$(\hat{\sigma}_i^{(k)},\hat{u}_i^{(k)},\hat{v}_i^{(k)})$ of
the desired singular triplets. The $\hat{v}^{(k)}\in \mathcal{V}^{(k)}$ and
$\hat{u}^{(k)}\in \mathcal{U}^{(k)}$
are the right and left Ritz vectors that approximate the right and
left singular vectors of $A$, respectively.
\Cref{alg:subspace iteration} is an adaptation of the FEAST eigensolver to our
SVD problem.
Particularly, as we will show in the proof of
\Cref{thm:triplets convergence}, this algorithm yields
$A\hat{v}_{i}^{(k)}=\hat{\sigma}_i^{(k)}\hat{u}_{i}^{(k)}$ (cf.
\eqref{leftright}). This means
that,when judging the convergence, we only need to compute the lower part
$A^T\hat{u}_i^{(k)}-\hat{\sigma}_i^{(k)}\hat{v}_i^{(k)}$
of the corresponding residual \eqref{resnorm} of an approximate singular
triplet, i.e., Ritz approximation or triplet,
$(\hat{\sigma}_i^{(k)},\hat{u}_i^{(k)},\hat{v}_i^{(k)})$.

\begin{algorithm}
    \caption{The basic FEAST SVDsolver: Subspace iteration on the approximate spectral projector $P$ for the partial SVD of $A$.}
    \label{alg:subspace iteration}
    \begin{algorithmic}[1]
        \REQUIRE{The matrix $A$, the interval $[a,b]$, the approximate
            spectral projector
            $P$, a $p$-dimensional subspace $\mathcal{V}^{(0)}$
            with $p\geq n_{sv}$, and $k=0$.}
        \ENSURE{$n_{sv}$ converged Ritz triplets $(\hat{\sigma}^{(k)},\hat{u}^{(k)},\hat{v}^{(k)})$.}
        \WHILE{not converged}
        \STATE{$k\leftarrow k+1$.}
        \STATE{Construct the right searching subspace: $\mathcal{V}^{(k)}=P\mathcal{V}^{(k-1)}$,
        and the left searching subspace $\mathcal{U}^{(k)}=A\mathcal{V}^{(k)}$.}
        \STATE\label{algstep:RR}{The Rayleigh--Ritz projection:
        find $\hat{u}^{(k)}\in \mathcal{U}^{(k)},\hat{v}^{(k)}\in \mathcal{V}^{(k)}, \hat\sigma^{(k)}\geq 0$
            with $\|\hat{u}^{(k)}\|=\|\hat{v}^{(k)}\|=1$ satisfying
            $A\hat{v}^{(k)}-\hat{\sigma}^{(k)}\hat{u}^{(k)}\perp \mathcal{U}^{(k)}, A^T\hat{u}^{(k)}-\hat{\sigma}^{(k)}\hat{v}^{(k)}\perp \mathcal{V}^{(k)}$.}
        \STATE{Compute the residual norms $\|r\|$, defined by \eqref{resnorm}, of $(\hat{\sigma}^{(k)},\hat{u}^{(k)},\hat{v}^{(k)})$ for all
        the $\hat{\sigma}^{(k)}\in [a,b]$.}
        \ENDWHILE
    \end{algorithmic}
\end{algorithm}

If $P=P_{S}$ defined by \eqref{ps1} and the subspace dimension $p=n_{sv}$,
then under the condition that the initial subspace $\mathcal{V}^{(0)}$ is not deficient in ${\rm span}\{V_{in},V_{ab}\}$,
\Cref{alg:subspace iteration} finds the $n_{sv}$ desired singular triplets
in one iteration since $\mathcal{V}^{(1)}
    ={\rm span}\{V_{in},V_{ab}\}$ and $\mathcal{U}^{(1)}$ are
the exact right and left singular subspaces of $A$ associated with all the
$\sigma\in [a,b]$.

The following lemma is about how to estimate the trace of a SPSD matrix by Monte--Carlo simulation \cite{avron2011randomized,Cortinovis2021onrandom}.

\begin{lemma}\label{lem:stochastic estimation}
    Let $P$ be an $n\times n$ SPSD matrix. Define $H_{M}=\frac{1}{M}\sum_{i=1}^{M}z_i^T P z_i$,
    where the components $z_{ij}$ of the random vectors $z_i$ are
    independent and identically distributed Rademacher random variables, i.e.,
    ${\rm Pr}(z_{ij}=1)={\rm Pr}(z_{ij}=-1)=\frac{1}{2}$. Then the expectation
    ${\rm E}(H_M)={\rm tr}(P)$ and variance ${\rm Var}(H_M)
    =\frac{2}{M}\bigl(\|P\|_F^2-\sum_{i=1}^{n}P_{ii}^2\bigr)$.
    Moreover,  ${\rm Pr}(|H_M-{\rm tr}(P)|\geq \epsilon \ {\rm tr}(P))\leq \delta$ for $ M \geq 8\epsilon^{-2}(1+\epsilon)\ln(\frac{2}{\delta})\|P\|/{\rm tr}(P)$.
\end{lemma}

This lemma will be exploited later to estimate $n_{sv}$ and
determine the subspace dimension $p\geq n_{sv}$ reliably
in our CJ-FEAST SVDsolver.

\section{The Chebyshev--Jackson series expansion of a specific step function}
\label{sec: Chebyshev series convergence}
For an interval $[a,b]\subset [-1,1]$,
define the step function
\begin{equation}\label{hdef}
    h(x)=
    \begin{cases}
        1, \quad x\in (a,b),                       \\
        \frac{1}{2}, \quad x\in \{a, b\}, \\
        0, \quad x\in [-1,1]\setminus [a,b],
    \end{cases}
\end{equation}
where $a$ and $b$ are the discontinuity points of $h(x)$, and
$h(a)=h(b)=\frac{1}{2}$ equal the means of respective right and
left limits:
\begin{equation*}
    \frac{h(a+0)+h(a-0)}{2}=\frac{h(b+0)+h(b-0)}{2}=\frac{1}{2}.
\end{equation*}
Suppose that $h(x)$ is approximately expanded as the Chebyshev--Jackson polynomial series
of degree $d$:
\begin{equation}\label{Chebyshevseries}
    h(x)\approx \psi_d(x)=\frac{c_0}{2}+\sum\limits_{j=1}^{d}\rho_{j,d}c_j T_{j}(x),
\end{equation}
where $T_j(x)$ is the $j$-degree Chebyshev polynomial of the first kind \cite{mason2002chebyshev}:
\begin{align*}
    T_0(x)=1, \ T_1(x)=x, \quad T_{j+1}(x)=2xT_{j}(x)-T_{j-1}(x), \ j\geq 1,
\end{align*}
the Fourier coefficients
\begin{equation}\label{coefficients}
    c_j=\begin{cases}
        \frac{2}{\pi}(\arccos(a)-\arccos(b)),\quad j = 0, \\
        \frac{2}{\pi}\bigl(\frac{\sin(j\arccos(a))-\sin(j\arccos(b))}{j}\bigr),\quad j=1,2,\dots,d,
    \end{cases}
\end{equation}
and the Jackson damping factors (cf. \cite{di2016efficient,jay1999electronic})
\begin{equation}\label{dampfactor}
    \rho_{j,d}=\frac{(d+2-j)\sin(\frac{\pi}{d+2})\cos(\frac{j\pi}{d+2})+
        \cos(\frac{\pi}{d+2})\sin(\frac{j\pi}{d+2})}{(d+2)\sin\frac{\pi}{d+2}}.
\end{equation}
We can also write $\rho_{j,d}$ as
\begin{equation}\label{dampfactororigin}
    \rho_{j,d}=2\sum_{\iota=0}^{d-j}t_{\iota}t_{\iota+j}, \quad j=0,1,\dots,d
\end{equation}
with
\begin{equation}\label{Deftl}
    t_{\iota}=\frac{\sin(\frac{\iota+1}{d+2}\pi)}
    {\sqrt{2\sum_{\iota=0}^{d}\sin^2(\frac{\iota+1}{d+2}\pi)}}, \quad  \iota=0,1,\dots,d;
\end{equation}
see \cite[Section 1.1.2]{rivlin1981introduction}.

Define the function $g(\theta)$ with period $2\pi$:
\begin{equation}\label{gdef}
    g(\theta):=h(\cos\theta).
\end{equation}
Then $g(\theta)$ is an even step function and
\begin{equation}\label{gthetadef}
    g(\theta)=
    \begin{cases}
        1, \quad \theta\in (\beta,\alpha)\cup (-\alpha,-\beta),       \\
        \frac{1}{2}, \quad \theta\in \{-\alpha,-\beta,\beta,\alpha\}, \\
        0, \quad \theta\in [-\pi,\pi]\setminus ([\beta,\alpha]\cup [-\alpha,-\beta]),
    \end{cases}
\end{equation}
where $\alpha=\arccos(a)$ and $\beta=\arccos(b).$
Define the trigonometric polynomial
\begin{equation}\label{qdef}
    q_{d}(\theta):= \psi_{d}(\cos\theta)=\frac{c_0}{2}+\sum_{j=1}^{d}\rho_{j,d}c_j\cos(j\theta).
\end{equation}

Lemma 1.4 of \cite[Section 1.1.2]{rivlin1981introduction} proves that if $s(\theta)$
is continuous on $\theta\in [-\pi,\pi]$ and has period $2\pi$ then
\begin{align*}
     & \frac{1}{2\pi}\int_{-\pi}^{\pi}s(\tau)d\tau +
     \sum_{j=1}^{d}\rho_{j,d}\biggl(\frac{\cos(j\theta)}
     {\pi}\int_{-\pi}^{\pi}s(\tau)\cos(j\tau)d\tau + \\
     &\frac{\sin(j\theta)}{\pi}\int_{-\pi}^{\pi}s(\tau)\sin(j\tau)d\tau \biggr)
      = \frac{1}{\pi}\int_{-\pi}^{\pi}s(\tau+\theta)\biggl
     (\frac{1}{2}+\sum_{j=1}^{d}\rho_{j,d}\cos(j\tau)\biggr)d\tau.
\end{align*}
The above equality obviously holds when $s(\tau)$ is replaced by
our step function $g(\tau)$ defined by \eqref{gthetadef},
which is piecewise continuous and has period $2\pi$.
Since $g(\tau)$ and $\sin(j\tau)$ are even and odd functions,
respectively, we obtain
\begin{equation*}
    \frac{1}{\pi}\int_{-\pi}^{\pi}g(\tau)\cos(j\tau)d\tau = c_j,
    \quad \frac{1}{\pi}\int_{-\pi}^{\pi}g(\tau)\sin(j\tau)d\tau=0, \quad j=0,1,\dots,d.
\end{equation*}
Consequently, we have proved the following lemma, which
indicates that $q_d(\theta)$ is the convolution of $g(\theta)$ and some
function $u_d(\theta)$ over the interval $[-\pi,\pi]$.

\begin{lemma}
    Let $g(\theta)$ and $q_{d}(\theta)$ be defined as \eqref{gdef} and \eqref{qdef}, respectively. Then
    \begin{equation}\label{qd}
        q_{d}(\theta)=\frac{1}{\pi}\int_{-\pi}^{\pi}g(\tau+\theta)u_{d}(\tau)d\tau,
    \end{equation}
    where
    \begin{equation}\label{udef}
        u_{d}(\tau)=\frac{1}{2}+\sum\limits_{j=1}^{d}\rho_{j,d}\cos(j\tau).
    \end{equation}
\end{lemma}

\begin{theorem}\label{thm:Nonnegative}
    For $\theta\in \mathbb{R}$, it holds that $q_{d}(\theta)\in [0,1]$.
\end{theorem}

\begin{proof}
    By \eqref{dampfactororigin}, it is known from  \cite[Section 1.1.2]{rivlin1981introduction} that
    \begin{equation*}
        u_{d}(\tau)=\sum_{\iota=0}^{d}t_{\iota}^2+\sum_{j=1}^{d}
        \biggl(\biggl(2\sum_{\iota=0}^{d-j}t_{\iota}t_{\iota+j} \biggr)\cos(j\tau) \biggr)
        =\biggl(\sum\limits_{\iota=0}^{d}t_{\iota}\mathrm{e}^{\mathrm{i}\iota\tau}\biggr)
        \biggl(\sum_{\iota=0}^{d}t_{\iota}\mathrm{e}^{-\mathrm{i}\iota\tau}\biggr)
        \geq 0,
    \end{equation*}
    where $t_{\iota},\ \iota=0,1,\ldots,d,$ are defined by \eqref{Deftl},
    $\mathrm{i}$ is the imaginary unit, and $\mathrm{e}$ is the natural constant.
    Since $g(\theta)\geq 0$, from \eqref{qd} we have
    $q_{d}(\theta)\geq 0$.
    On the other hand,
    \begin{equation}\label{udint}
        \int_{-\pi}^{\pi}u_{d}(\tau)d\tau =
        \frac{1}{2}\int_{-\pi}^{\pi}d\tau+\sum\limits_{j=1}^{d}\rho_{j,d}
        \int_{-\pi}^{\pi}\cos(j\tau)d\tau = \pi.
    \end{equation}
    Therefore,
    \begin{equation*}
        q_{d}(\theta)=\frac{1}{\pi}\int_{-\pi}^{\pi}g(\tau+\theta)u_{d}(\tau)d\tau
        \leq \frac{1}{\pi}\int_{-\pi}^{\pi}u_{d}(\tau) d\tau=1.
    \end{equation*} \qed
\end{proof}

Next we establish {\em quantitative} results on how
fast $q_d(\theta)$ converges to $g(\theta)$ in the {\em pointwise} sense.
We first consider the case that $\theta\neq \alpha,\beta$.

\begin{theorem}\label{thm:pointwise convergence}
    Let $g(\theta)$ and $q_{d}(\theta)$ be defined as \eqref{gdef} and \eqref{qdef}, respectively.
    For $\theta\in [0,\pi]$, $\theta\not= \alpha,\beta$ and $\alpha>\beta$,
    define
    \begin{equation*}
        \Delta_{\theta}=\min\{|\theta-\alpha|,|\theta-\beta|\}.
    \end{equation*}
    Then for $d\geq 2$ we have
    \begin{equation}\label{bound-3}
        |q_{d}(\theta)-g(\theta)|\leq \frac{\pi^6}{2(d+2)^3\Delta_{\theta}^4}.
    \end{equation}
\end{theorem}

\begin{proof}
    According to \eqref{gdef} and \eqref{gthetadef}, we have
    \begin{equation}\label{presult}
        g(\tau)=g(\tau - 2\pi)=0 \mbox{ \ for \ } \pi < \tau < 2\pi-\alpha.
    \end{equation}
    For any given $\theta\in [0,\pi]$, define the function
    \begin{equation}\label{defF}
        F_{\theta}(\tau)=
        \begin{cases}
            \frac{g(\tau+\theta)-g(\theta)}{\tau^4}, & \tau\not=0, \\
            0,                                       & \tau=0.
        \end{cases}
    \end{equation}
    We classify $\theta\in [0,\pi]$ as $\theta \in [0,\beta)$,
    $\theta \in (\beta,\alpha)$ and $\theta \in (\alpha,\pi]$. Note that
    \begin{itemize}
        \item[] if $\theta \in [0,\beta)$ then $\Delta_{\theta}=\beta-\theta$ and $\tau+\theta \in (-\beta,\beta)$ for $|\tau|< \Delta_{\theta}$,
        \item[] if $\theta \in (\beta,\alpha)$ then $\Delta_{\theta}={\min}\{\theta-\beta,\alpha-\theta\}$ and $\tau+\theta \in (\beta,\alpha)$  for $|\tau|< \Delta_{\theta}$,
        \item[] if $\theta \in (\alpha,\pi]$  then  $\Delta_{\theta}=\theta-\alpha$ and $\tau+\theta \in (\alpha,2\pi-\alpha)$ for $|\tau|< \Delta_{\theta}$.
    \end{itemize}
    Therefore, for any given $\theta\in [0,\pi]$ and $\theta\neq \alpha,\beta$, if
    $|\tau|< \Delta_{\theta}$, then by \eqref{gthetadef} we have
    $g(\tau+\theta)=g(\theta)$. As a result, we obtain
    \begin{equation*}
        F_{\theta}(\tau)=0  \mbox{ \ for \ } |\tau| < \Delta_{\theta}.
    \end{equation*}
    On the other hand, since $|g(\tau+\theta)-g(\theta)|\leq 1$ for
    $|\tau| \geq \Delta_{\theta}$, we have
    \begin{equation*}
        |F_{\theta}(\tau)|\leq \frac{1}{\Delta_{\theta}^4}  \mbox{ \ for \ } |\tau| \geq \Delta_{\theta}.
    \end{equation*}
    Combining the above two relations yields
    \begin{equation*}
        |F_{\theta}(\tau)|\leq \frac{1}{\Delta_{\theta}^4}  \mbox{ \ for \ }  \tau\in\mathbb{R}.
    \end{equation*}
    Exploiting \eqref{udint}, we obtain
    \begin{equation*}
        g(\theta)=\frac{1}{\pi}\int_{-\pi}^{\pi} g(\theta)u_d(\tau)d\tau.
    \end{equation*}
    Therefore, it follows from \eqref{defF} that
    \begin{align}
        |q_{d}(\theta)-g(\theta)| & =\left|\frac{1}{\pi}\int_{-\pi}^{\pi}(g(\tau+\theta)-g(\theta))
        u_{d}(\tau)d\tau\right| \notag\\
      &  \leq \frac{1}{\pi}\int_{-\pi}^{\pi}|F_{\theta}(\tau)|\tau^4u_{d}(\tau)d\tau         \notag\\
        & \leq \frac{1}{\pi}\int_{-\pi}^{\pi}\frac{\tau^4}{\Delta_{\theta}^4}u_{d}(\tau)d\tau. \label{delta4}
    \end{align}
    Making use of the inequality
    \begin{equation*}
        \left|\frac{\tau}{2}\right| \leq \frac{\pi}{2} \left|\sin(\frac{\tau}{2})\right| \ \mbox{for} \ |\tau| \leq \pi
    \end{equation*}
    (cf. \cite[Lemma 1.5, Section 1.1.2]{rivlin1981introduction}),
    we obtain
    \begin{equation*}
        \tau^4 \leq \pi^4 \sin^4(\frac{\tau}{2}) =  \pi^4 \bigl(\frac{1-\cos(\tau)}{2}\bigr)^2 = \frac{\pi^4}{8}(3-4\cos(\tau)+\cos(2\tau)) \ \mbox{for} \ |\tau| \leq \pi
    \end{equation*}
    and
    \begin{equation}\label{tau4power}
        \frac{1}{\pi}\int_{-\pi}^{\pi}\tau^4u_{d}(\tau)d\tau
        \leq \frac{\pi^3}{8}\int_{-\pi}^{\pi}(3-4\cos(\tau)+\cos(2\tau))u_{d}(\tau)d\tau .
    \end{equation}
    It follows from
    \begin{equation*}
        \frac{1}{\pi}\int_{-\pi}^{\pi} \cos(i\tau)\cos(j\tau)d\tau = \delta_{i,j} = \begin{cases}
            1, \ i= j, \\
            0, \ i \neq j,
        \end{cases} \ \mbox{for} \ i, j \geq 1
    \end{equation*}
    and \eqref{udef} that
    \begin{equation*}
        \int_{-\pi}^{\pi} \cos(k\tau)u_{d}(\tau)d\tau =\rho_{k,d} \pi \ \ \mbox{for}\ \ k=1,2.
    \end{equation*}
Therefore, combining the above relation, \eqref{tau4power}, \eqref{udef} and \eqref{dampfactor},
we obtain
    \begin{align*}
        \frac{1}{\pi}\int_{-\pi}^{\pi}\tau^4u_{d}(\tau)d\tau
        &\leq \frac{\pi^4}{8} (3 - 4\rho_{1,d} + \rho_{2,d}) \\
        &=\frac{\pi^4}{8} \left(3 - 4\cos(\frac{\pi}{d+2}) + \frac{(2d+2)\cos^2(\frac{\pi}{d+2})-d}{d+2}\right) \\
        &=\frac{\pi^4}{4(d+2)} \left(1-\cos(\frac{\pi}{d+2})\right) \left(d+3 - (d+1)\cos(\frac{\pi}{d+2})\right).
    \end{align*}
    Since
    \begin{align*}
        1-\cos(\frac{\pi}{d+2}) &= 2 \sin^2(\frac{\pi}{2d+4}) \leq 2\left(\frac{\pi}{2d+4}\right)^2 = \frac{\pi^2}{2(d+2)^2}
    \end{align*}
    and
    \begin{align}
        d+3 - (d+1)\cos(\frac{\pi}{d+2}) &= 2 + (d+1)\left(1-\cos(\frac{\pi}{d+2})\right) \notag\\
        &= 2 + 2(d+1)  \sin^2(\frac{\pi}{2d+4}) \notag \\
        &\leq 2 + \frac{(d+1)\pi^2}{2(d+2)^2} <4,\label{factor}
    \end{align}
    we get
    \begin{equation}\label{udupperbound}
        \frac{1}{\pi}\int_{-\pi}^{\pi}\tau^4u_{d}(\tau)d\tau \leq \frac{\pi^6}{2(d+2)^3}.
    \end{equation}
    The above relation and \eqref{delta4} prove \eqref{bound-3}. \qed
\end{proof}

We comment that bound \eqref{factor} is approximately equal to 2 for a modestly sized $d$, e.g., say 20, so that bound \eqref{bound-3} is approximately reduced by half as $d$ increases.

If $\theta$ is equal to the discontinuity point
$\alpha$ or $\beta$, we need to make a separate analysis. We next prove
how $q_d(\alpha)$ and $q_d(\beta)$ converge to $g(\alpha)=g(\beta)=\frac{1}{2}$.

\begin{theorem}\label{thm:convergence at dcon}
    Let $g(\theta)$ and $q_{d}(\theta)$ be defined as \eqref{gdef} and \eqref{qdef}, respectively.
    Then for $\alpha,\beta\in (0,\pi)$, $\alpha>\beta$ and $d\geq 2$ it holds that
    \begin{align}
        |q_d(\alpha)-g(\alpha)| & \leq \frac{\pi^6}{2(d+2)^3} \max \biggl\{ \frac{1}{(2\pi-2\alpha)^4}, \frac{1}{(\alpha-\beta)^4} \biggr\},\label{error for alpha} \\
        |q_d(\beta)-g(\beta)|   & \leq \frac{\pi^6}{2(d+2)^3}\max \biggl\{\frac{1}{(2\beta)^4},
        \frac{1}{(\alpha-\beta)^4}\biggr\}\label{error for beta}.
    \end{align}
\end{theorem}

\begin{proof}
    We first consider the case $\theta=\alpha$. Define the functions
    \begin{equation*}
        F_{\alpha}(\tau)=
        \begin{cases}
            \frac{g(\tau+\alpha)}{\tau^4}, & \tau > 0, \\
            0,                             & \tau=0
        \end{cases}
        \ \mbox{and} \
        G_{\alpha}(\tau)=
        \begin{cases}
            \frac{g(\tau+\alpha)-1}{\tau^4}, & \tau <0, \\
            0,                               & \tau=0.
        \end{cases}
    \end{equation*}
    For $\tau\in (0,2\pi-2\alpha)$, we have
    $\tau+\alpha\in (\alpha,2\pi-\alpha)$.
    Therefore, from \eqref{gthetadef} and \eqref{presult},  we obtain $g(\tau+\alpha)=0$,
    showing that
    \begin{equation*}
        F_{\alpha}(\tau)=0  \mbox{ \ for \ } 0<\tau < 2\pi-2\alpha.
    \end{equation*}
    On the other hand, $0\leq g(\tau+\alpha)\leq 1$ means that
    \begin{equation*}
        0 \leq F_{\alpha}(\tau) \leq \frac{1}{(2\pi-2\alpha)^4}  \mbox{ \ for \ } \tau \geq 2\pi-2\alpha.
    \end{equation*}
    Combining the above two relations yields
    \begin{equation*}
        0 \leq F_{\alpha}(\tau) \leq \frac{1}{(2\pi-2\alpha)^4}  \mbox{ \ for \ }  \tau \geq 0.
    \end{equation*}

    For $\tau\in (\beta-\alpha,0)$, we have
    $\tau+\alpha\in (\beta,\alpha)$. Therefore, from
    \eqref{gthetadef}, we have $g(\tau+\alpha)=1$, leading to
    \begin{equation*}
        G_{\alpha}(\tau)=0  \mbox{ \ for \ } \beta-\alpha < \tau <0.
    \end{equation*}
    On the other hand, by $-1\leq g(\tau+\alpha)-1\leq 0$, we have
    \begin{equation*}
        -\frac{1}{(\alpha-\beta)^4} \leq G_{\alpha}(\tau) \leq 0  \mbox{ \ for \ } \tau \leq \beta-\alpha.
    \end{equation*}
    The above two relations show that
    \begin{equation*}
        -\frac{1}{(\alpha-\beta)^4} \leq G_{\alpha}(\tau) \leq 0 \mbox{ \ for \ }  \tau \leq 0.
    \end{equation*}
    Since $u_d(\tau)$ is an even function and $\int_{-\pi}^{\pi}u_{d}(\tau)d\tau=\pi$ (cf. \eqref{udint}),
    we have
    \begin{equation}\label{udpi/2}
        \int_{-\pi}^{0}u_{d}(\tau)d\tau=\int_{0}^{\pi}u_{d}(\tau)d\tau=\frac{\pi}{2}.
    \end{equation}
     Keep in mind $g(\alpha)=\frac{1}{2}$. Therefore,
    \begin{align*}
        q_d(\alpha)-\frac{1}{2} & =\frac{1}{\pi}\int_{-\pi}^{\pi}g(\tau+\alpha)u_{d}(\tau)d\tau-\frac{1}{2}                                                              \\
                                & =\frac{1}{\pi}\int_{0}^{\pi}g(\tau+\alpha)u_{d}(\tau)d\tau+\frac{1}{\pi}\int_{-\pi}^{0}g(\tau+\alpha)u_{d}(\tau)d\tau-\frac{1}{2}      \\
                                & =\frac{1}{\pi}\int_{0}^{\pi}g(\tau+\alpha)u_{d}(\tau)d\tau+\frac{1}{\pi}\int_{-\pi}^{0}(g(\tau+\alpha)-1)u_{d}(\tau)d\tau              \\
                                & =\frac{1}{\pi}\int_{0}^{\pi}F_{\alpha}(\tau)\tau^4 u_{d}(\tau)d\tau+\frac{1}{\pi}\int_{-\pi}^{0}G_{\alpha}(\tau)\tau^4 u_{d}(\tau)d\tau.
    \end{align*}
    Exploiting \eqref{udupperbound}, we obtain
    \begin{align*}
         & 0\leq \frac{1}{\pi}\int_{0}^{\pi}F_{\alpha}(\tau)\tau^4 u_{d}(\tau)d\tau \leq \frac{1}{(2\pi-2\alpha)^4}\frac{1}{\pi}\int_{0}^{\pi}\tau^4 u_{d}(\tau)d\tau \leq \frac{1}{(2\pi-2\alpha)^4}\frac{\pi^6}{2(d+2)^3},      \\
         & 0\geq \frac{1}{\pi}\int_{-\pi}^{0}G_{\alpha}(\tau)\tau^4 u_{d}(\tau)d\tau \geq -\frac{1}{(\alpha-\beta)^4}\frac{1}{\pi}\int_{-\pi}^{0}\tau^4 u_{d}(\tau)d\tau \geq -\frac{1}{(\alpha-\beta)^4}\frac{\pi^6}{2(d+2)^3},
    \end{align*}
    which proves \eqref{error for alpha}.

    Now we consider the case $\theta = \beta$. Define the functions
    \begin{equation*}
        F_{\beta}(\tau)=
        \begin{cases}
            \frac{g(\tau+\beta)-1}{\tau^4}, & \tau > 0, \\
            0,                              & \tau=0
        \end{cases}
        \ \mbox{and} \
        G_{\beta}(\tau)=
        \begin{cases}
            \frac{g(\tau+\beta)}{\tau^4}, & \tau <0, \\
            0,                            & \tau=0.
        \end{cases}
    \end{equation*}
    For $\tau\in (0,\alpha-\beta)$, we have $\tau+\beta\in (\beta,\alpha)$.
    Therefore, by \eqref{gthetadef}, we obtain $g(\tau+\beta)=1$, so that
    \begin{equation*}
        F_{\beta}(\tau)=0  \mbox{ \ for \ } 0<\tau < \alpha-\beta.
    \end{equation*}
    On the other hand, by $-1\leq g(\tau+\beta)-1\leq 0$, we obtain
    \begin{equation*}
        -\frac{1}{(\alpha-\beta)^4} \leq F_{\beta}(\tau) \leq 0  \mbox{ \ for \ } \tau \geq \alpha-\beta.
    \end{equation*}
    Combining the above two relations yields
    \begin{equation*}
        -\frac{1}{(\alpha-\beta)^4} \leq F_{\beta}(\tau) \leq 0   \mbox{ \ for \ }  \tau \geq 0.
    \end{equation*}
    Since $\tau\in (-2\beta,0)$ means that $\tau+\beta\in (-\beta,\beta)$,
    by \eqref{gthetadef} we have $g(\tau+\beta)=0$, leading to
    \begin{equation*}
        G_{\beta}(\tau)=0  \mbox{ \ for \ } -2\beta <\tau < 0.
    \end{equation*}
    On the other hand, since $0\leq g(\tau+\beta)\leq 1$, we have
    \begin{equation*}
        0 \leq G_{\beta}(\tau) \leq \frac{1}{(2\beta)^4}  \mbox{ \ for \ } \tau \leq -2\beta.
    \end{equation*}
    Therefore,
    \begin{equation*}
        0\leq G_{\beta}(\tau) \leq \frac{1}{(2\beta)^4} \mbox{ \ for \ }  \tau \leq 0.
    \end{equation*}
    Keep in mind $g(\beta)=\frac{1}{2}$.
    As done for $q_d(\alpha)-\frac{1}{2}$, we have
    \begin{align*}
        q_d(\beta)-\frac{1}{2} & =\frac{1}{\pi}\int_{-\pi}^{\pi}g(\tau+\beta)u_{d}(\tau)d\tau-\frac{1}{2}                                                             \\
                               & =\frac{1}{\pi}\int_{0}^{\pi}F_{\beta}(\tau)\tau^4 u_{d}(\tau)d\tau+
                               \frac{1}{\pi}\int_{-\pi}^{0}G_{\beta}(\tau)\tau^4 u_{d}(\tau)d\tau.
    \end{align*}
    By \eqref{udupperbound}, we have
    \begin{align*}
         & 0\geq \frac{1}{\pi}\int_{0}^{\pi}F_{\beta}(\tau)\tau^4 u_{d}(\tau)d\tau \geq -\frac{1}{(\alpha-\beta)^4}\frac{1}{\pi}\int_{0}^{\pi}\tau^4 u_{d}(\tau)d\tau \geq -\frac{1}{(\alpha-\beta)^4}\frac{\pi^6}{2(d+2)^3}, \\
         & 0\leq \frac{1}{\pi}\int_{-\pi}^{0}G_{\beta}(\tau)\tau^4 u_{d}(\tau)d\tau \leq  \frac{1}{(2\beta)^4}\frac{1}{\pi}\int_{-\pi}^{0}\tau^4 u_{d}(\tau)d\tau
         \leq \frac{1}{(2\beta)^4}\frac{\pi^6}{2(d+2)^3},
    \end{align*}
    which proves \eqref{error for beta}. \qed
\end{proof}

By definition \eqref{Chebyshevseries} of $\psi_d(x)$ and \eqref{qdef},
by taking $\theta=\arccos(x)$,
\Cref{thm:pointwise convergence} and \Cref{thm:convergence at dcon} show how
fast $\psi_{d}(x)$ {\em pointwise}
converges to $h(x)$ for $x\in [-1,1]$. They
indicate that the approximation errors are proportional to $\frac{1}{(d+2)^3}$,
that is, apart from a constant factor,
the convergence of $\psi_d(x)$ to $h(x)$ is as least as fast as
$\frac{1}{(d+2)^3}$ for $x\in [-1,1]$. Numerical experiments
have demonstrated that the optimal convergence rate is indeed
$\frac{1}{(d+2)^3}$ and cannot be improved, as shown below.

When assessing our a-priori bounds, we should point out that the bounds
may be large overestimates of the true errors, but
that there may be cases where the actual errors and their bounds
become close to each other when $d$ increases.
Possible overestimates of our bounds are not surprising,
since the bounds are established in the worst case and
the constants, apart from $\frac{1}{(d+2)^3}$,
are the largest possible. Our aim consists in justifying that the a-priori
bound indeed yields sharp estimates of the asymptotic {\em convergence rates}
even if the constant is large, that is, we are
concerned with the {\em insight} into the convergence rates.

Keep in mind the above.
We present an example to illustrate \eqref{bound-3}, \eqref{error for alpha} and \eqref{error for beta}.
Take $[a,b] = [-0.3, 0.5]\subset [-1,1]$ and the four points
$x= -0.4, -0.3, 0.1, 0.5$, of which $-0.4$ and $0.1$ are outside and inside  $[a,b]$, respectively. Note that $\alpha=\arccos(-0.3), \beta=\arccos(0.5), \theta=\arccos(x)$ for other $x\in [-1,1]$. For each of the four $x$,
we plot the true errors $|\psi_{d}(x)-h(x)|$ and error bounds \eqref{bound-3}
for $d=1,2,\ldots,10000$ in \Cref{fig:bound}. Clearly, the bounds reflect
the asymptotic rate $\frac{1}{(d+2)^3}$ precisely, and both the bounds and the true errors converge to zero in the same rates as $d$ increases. More precisely, for $x=-0.3$ and $0.5$, the bounds are quite accurate estimates
for the true errors within an approximate
multiple 100 all the while; but for $x=-0.4$ and $0.1$, the bounds
deviate from the true errors considerably, especially for $d$ small.
We can see from the figure that the errors have already reached $0.01\sim 0.1$ for a modest $d$.

\begin{figure}[tbhp]
    \centering
    \subfloat[$x=-0.4$]{\includegraphics[scale=0.4]{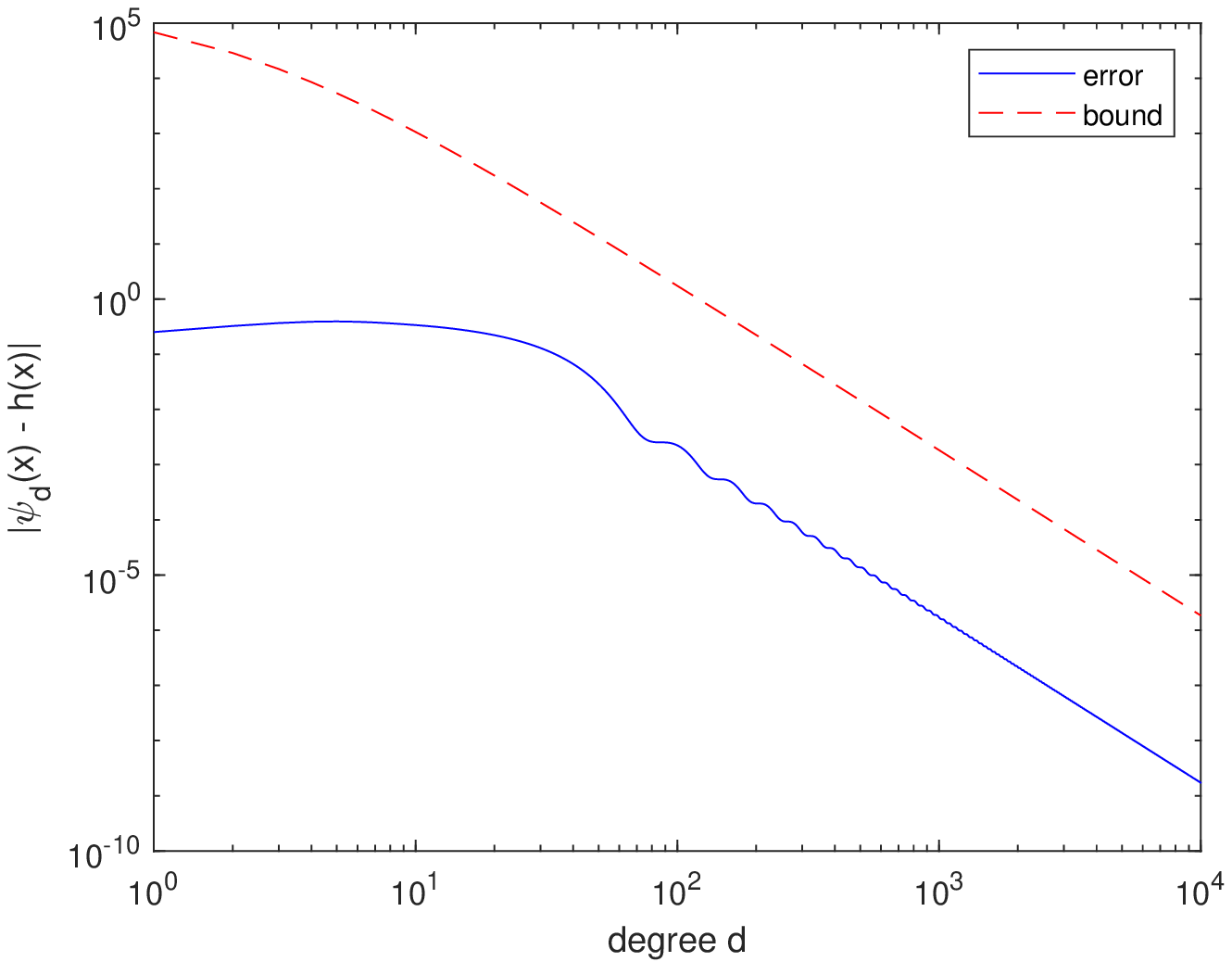}}
    \subfloat[$x=-0.3$]{\includegraphics[scale=0.4]{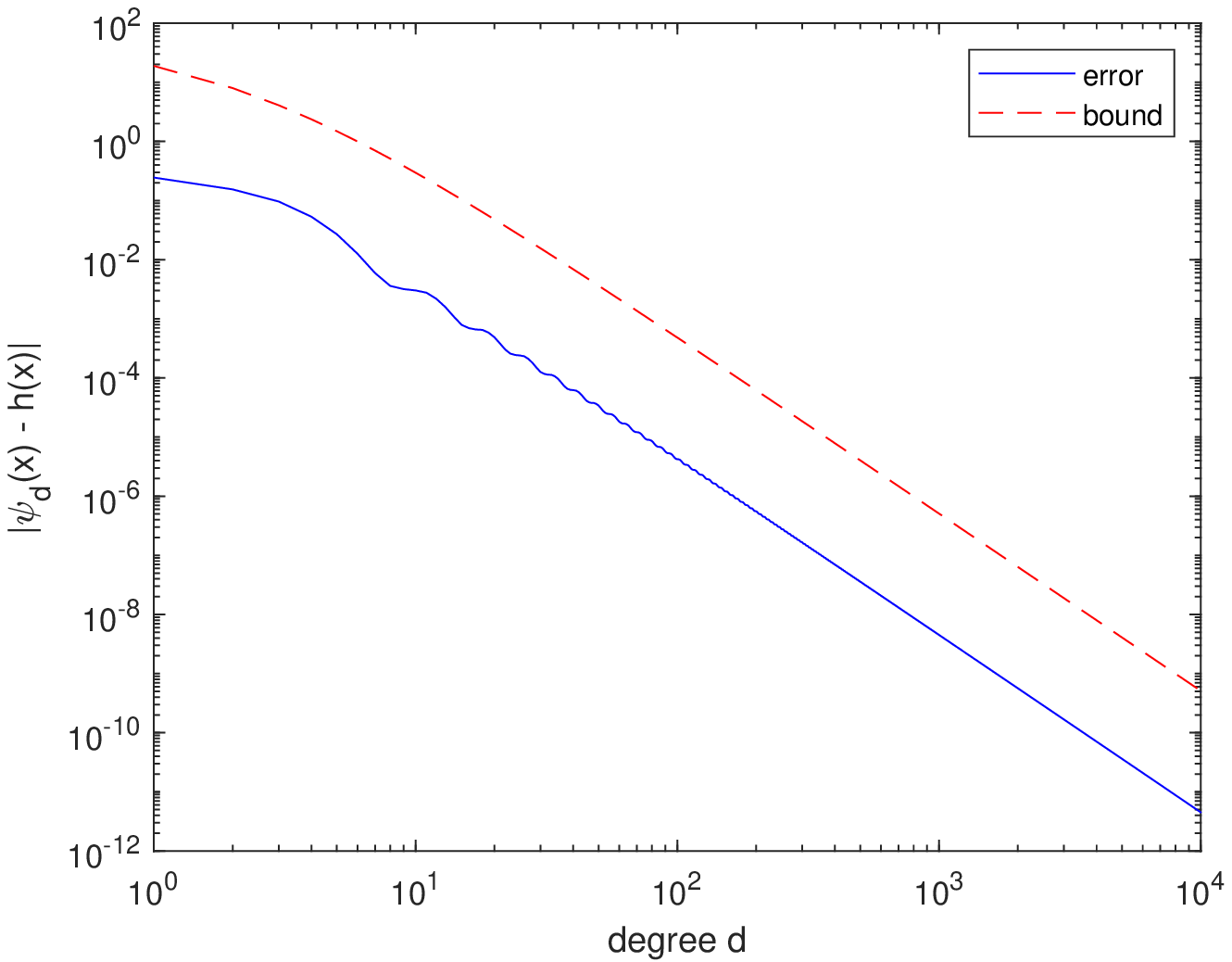}}

    \subfloat[$x=0.1$]{\includegraphics[scale=0.4]{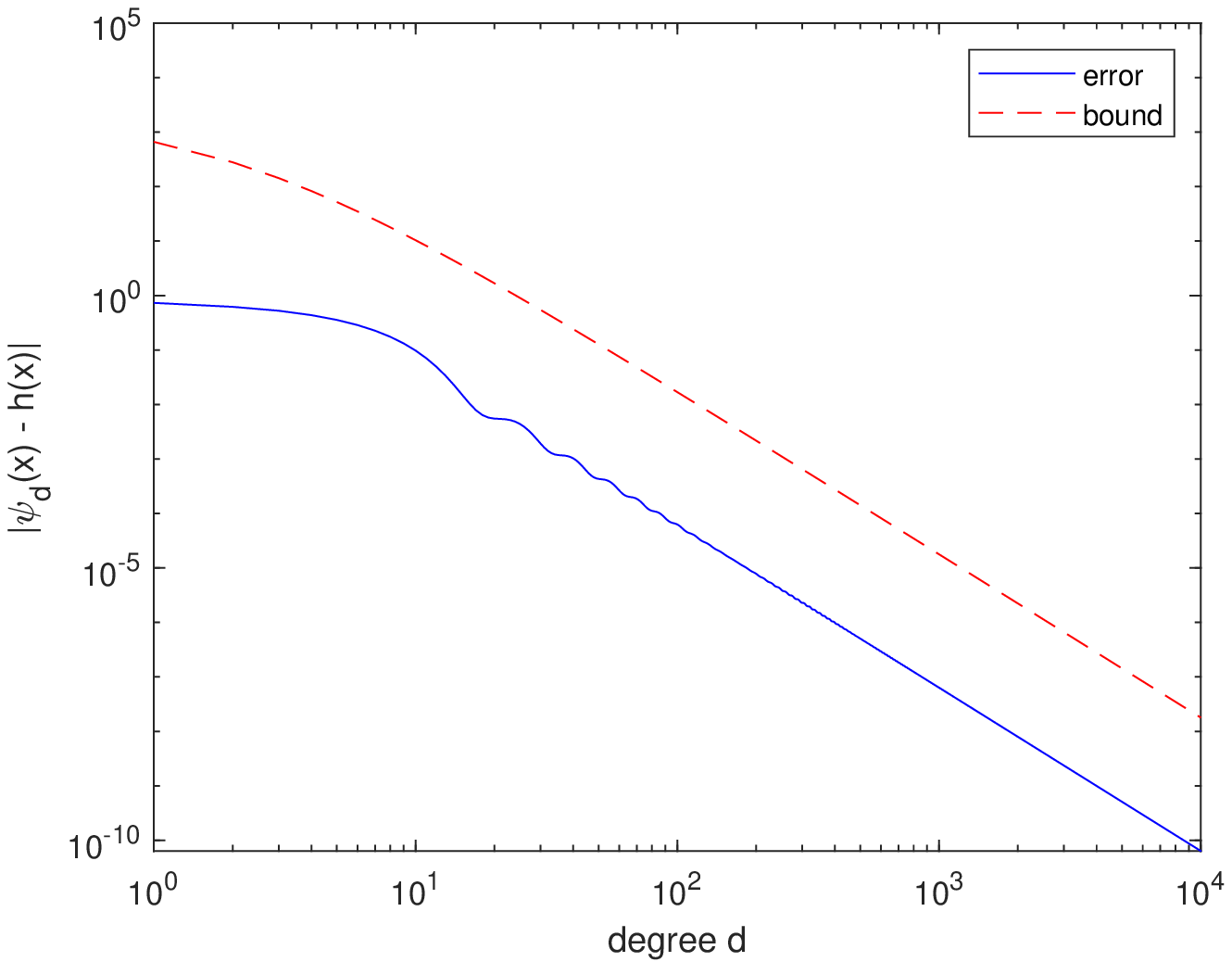}}
    \subfloat[$x=0.5$]{\includegraphics[scale=0.4]{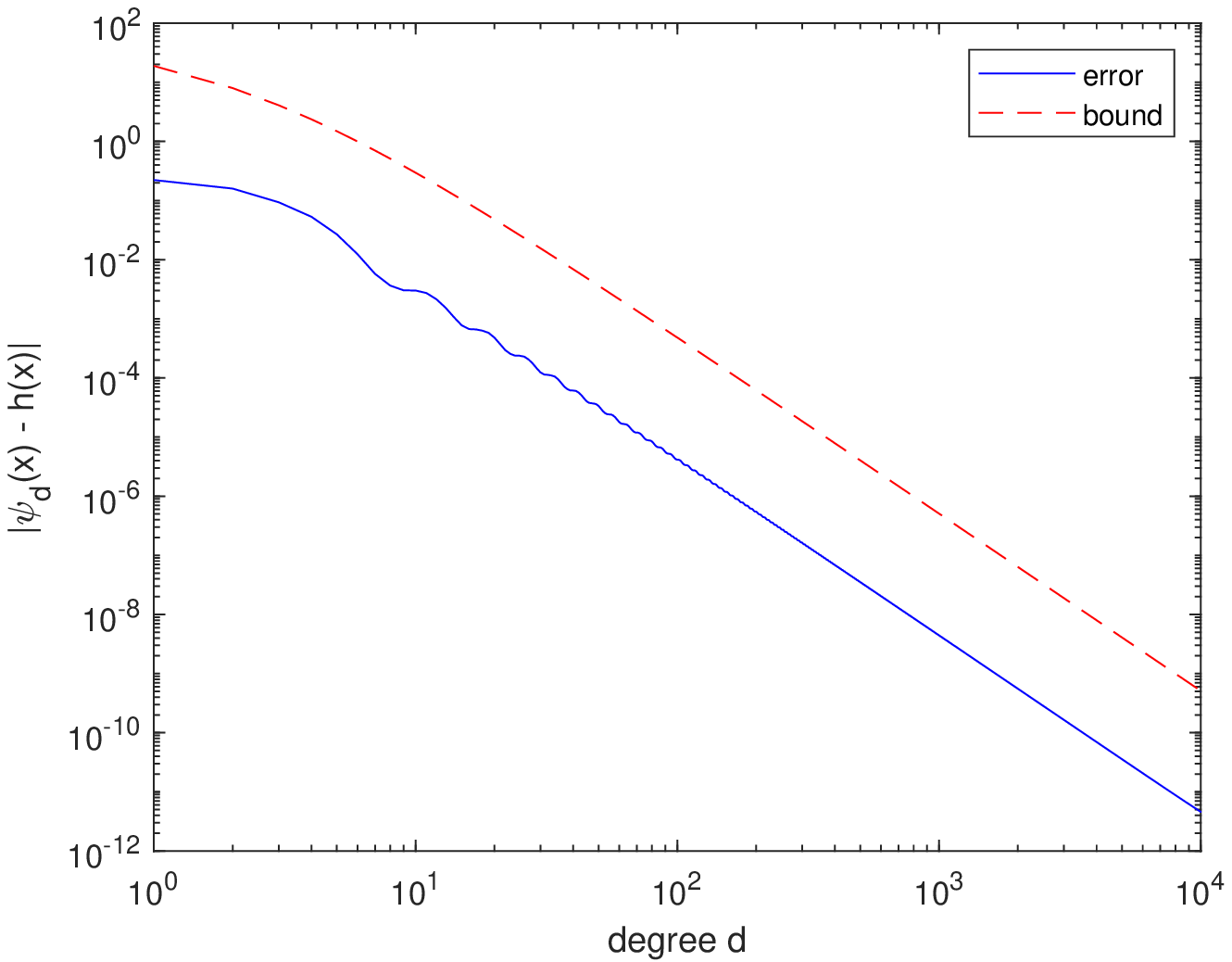}}
    \caption{True errors and error bounds.}
    \label{fig:bound}
\end{figure}

\section{The CJ-FEAST SVDsolver}
\label{sec: cross product method}

\subsection{Approximate spectral projector and its accuracy}\label{sec4.1}
We use the linear transformation
\begin{equation}\label{defl}
    l(x)=\frac{2x-\|A\|^2-\sigma_{\min}^2}{\|A\|^2-\sigma_{\min}^2}
\end{equation}
to map the spectrum interval $[\sigma_{\rm{min}}^2,\|A\|^2]$ of $S=A^TA$ to $[-1,1]$. We remind that,
to use the transformation in computation, it suffices to give rough estimates   for $\|A\|$ and $\sigma_{\min}$.
We can run a Lanczos, i.e., Golub--Kahan, bidiagonalization type
method on $A$ several steps, say $20\sim 30$, to estimate them
\cite{golub2013matrix,jia2003implicitly,jia2010refined}, which costs
very little compared to that of the CJ-FEAST SVDsolver.
For a given interval $[a,b]\subset [\sigma_{\min},\|A\|]$,
define the step function
\begin{equation*}
    h(x)=
    \begin{cases}
        1,\quad x\in (l(a^2),l(b^2)),  \\
        \frac{1}{2}, \quad x \in \{l(a^2), l(b^2)\}, \\
        0,\quad x\in [-1,1]\setminus [l(a^2),l(b^2)]
    \end{cases}
\end{equation*}
and the composite function
\begin{equation*}
    f(x)=h(l(x)).
\end{equation*}
Therefore,
\begin{equation}\label{deff}
    f(x)=
    \begin{cases}
        1,\quad x\in (a^2,b^2),                        \\
        \frac{1}{2}, \quad x\in \{a^2, b^2\}, \\
        0,\quad x\in [\sigma_{\min}^2,\|A\|^2]\setminus[a^2,b^2].
    \end{cases}
\end{equation}
Recall definition \eqref{ps1} of $P_S$. It follows from the
above and \eqref{ceigen} that
\begin{equation}\label{Pdef1}
    f(S)=Vf(\Sigma^2)V^T=P_{S}.
\end{equation}

\Cref{thm:pointwise convergence} and \Cref{thm:convergence at dcon} prove that $\psi_d(l(x))$ pointwise converges to $f(x)$.
Correspondingly, we
construct an approximate spectral projector
\begin{equation}\label{hatp}
    P=\psi_{d}(l(S))=\sum_{j=0}^d \rho_{j,d}c_j T_j(l(S)),
\end{equation}
whose eigenvector matrix is $V$ and eigenvalues are
$\gamma_i:=\psi_{d}(l(\sigma_i^2))$ with $\sigma_i, i=1,2,\dots,n$ being the singular values of $A$.
For convenience, $c_0$ in \eqref{hatp} corresponds to  $\frac{c_0}{2}$ in \eqref{Chebyshevseries}. We see that, given a basis matrix of the subspace
$\mathcal{V}^{(k-1)}$,
the unique action of $P$ in \Cref{alg:subspace iteration} is to form matrix-matrix products. We only need to store the coefficients $c_j,\rho_{j,d}, j=0,\dots,d$ without forming $P$ explicitly.
We describe the computation of Chebyshev--Jackson coefficients as \Cref{alg:Spectral Projector1}.

\begin{algorithm}
    \caption{The computation of Chebyshev--Jackson coefficients}
    \label{alg:Spectral Projector1}
    \begin{algorithmic}[1]
        \REQUIRE{The matrix $A$, the interval $[a, b]$, and the series degree $d$.}
        \ENSURE{$c_j,  \rho_{j,d}, j=0,\dots,d$.}
        \STATE{$\alpha=\arccos(l(a^2)), \quad   \beta=\arccos(l(b^2))$}
        \STATE{$\zeta=\frac{\pi}{d+2}$.}
        \FOR{$j=0,1,\dots,d$}
        \STATE{ $c_j=\begin{cases}
                    \frac{\alpha-\beta}{\pi},\quad j=0, \\
                    \frac{2}{\pi}\frac{\sin(j\alpha)-\sin(j\beta)}{j},\quad j>0,
                \end{cases}$
        $\rho_{j,d}=\frac{(d+2-j)\sin\zeta\cos(j\zeta)+\cos\zeta\sin(j\zeta)}{(d+2)\sin\zeta}$}.
        \ENDFOR
    \end{algorithmic}
\end{algorithm}

Next we estimate $\|P_{S}-P\|$ and the $\gamma_i$, which are key
quantities that critically affect
the convergence of the CJ-FEAST SVDsolver to be proposed and developed.

\begin{theorem}\label{thm:accuracyps1}
    Given the interval $[a,b]\subset [\sigma_{\min},\|A\|]$, let
    \begin{align*}
         & \alpha=\arccos(l(a^2)), \quad   \beta=\arccos(l(b^2)),                                               \\
         & \Delta_{il}=|\arccos(l(\sigma_{il}^2))-\alpha|, \quad \Delta_{ir}=|\arccos(l(\sigma_{ir}^2))-\beta|, \\
         & \Delta_{ol}=|\arccos(l(\sigma_{ol}^2))-\alpha|, \quad \Delta_{or}=|\arccos(l(\sigma_{or}^2))-\beta|,
    \end{align*}
    where $\sigma_{il},\ \sigma_{ir}$ and
    $\sigma_{ol},\ \sigma_{or}$ are the singular values of $A$ that
    are the closest to the ends $a$ and $b$ from inside and outside of $[a,b]$, respectively.
    Define
    \begin{equation}\label{dmin}
        \Delta_{\min}={\min}\{\Delta_{il},\Delta_{ir},\Delta_{ol},\Delta_{or}\}.
    \end{equation}
    Then
    \begin{equation}\label{Accuracy of projector1}
        \|P_{S}-P\| \leq \frac{\pi^6}{2(d+2)^3\Delta_{\min}^4}.
    \end{equation}
    Suppose that the singular values of $A$ in $[a,b]$ are
    $\sigma_1,\ldots,\sigma_{n_{sv}}$ with $\sigma_1,\ldots,\sigma_r$ in
    $(a,b)$ and $\sigma_{r+1},\ldots, \sigma_{n_{sv}}$
    equal to $a$ or $b$
    and those in $[\sigma_{\min},\|A\|]\setminus [a,b]$
    are $\sigma_{n_{sv}+1},\ldots,\sigma_n$, and label the eigenvalues
    $\gamma_i$ of $P,\ i=1,2,\ldots,r$, $i=r+1,\ldots,n_{sv}$
    and $i=n_{sv}+1,\dots,n$
    in decreasing order, respectively.
    If
        \begin{equation}\label{dsize}
        d \geq \frac{\sqrt[3]{2}\pi^2}{\Delta_{\min}^{4/3}}-2,
    \end{equation}
    then
    \begin{equation}\label{errorp}
        \|P_{S}-P\|<\frac{1}{4}
    \end{equation}
    and
    \begin{equation}\label{evhatp}
        1\geq \gamma_1  \geq \cdots \geq \gamma_r>\frac{3}{4}> \gamma_{r+1}\geq \cdots \geq\gamma_{n_{sv}}
                >\frac{1}{4}>\gamma_{n_{sv}+1}\geq \cdots\geq \gamma_n\geq 0.
    \end{equation}
\end{theorem}

\begin{proof}
    Since the eigenvalues of
    $P_{S}$ are
    \begin{equation*}
        f(\sigma_{i}^2)=h(l(\sigma_{i}^2))=
        \begin{cases}
            1, \quad \sigma_{i}\in (a,b),                      \\
            \frac{1}{2}, \quad \sigma_{i}=a \mbox{ or } b, \\
            0, \quad n_{sv}+1\leq i \leq n,
        \end{cases}
    \end{equation*}
    from \eqref{hatp} we obtain
    \begin{align}
        \|P_{S}-P\|     & =\|f(S)-\psi_d(l(S)) \|  =\|f(\Sigma^2)-\psi_d(l(\Sigma^2)) \|   \nonumber      \\
                        & =\max_{i=1,2,\ldots,n}|h(l(\sigma_i^2))-\psi_d(l(\sigma_i^2))|  \label{diffeig}    \\
                        & =\max_{i=1,2,\ldots,n}|h(\cos(\theta_i))-\psi_d(\cos(\theta_i))| , \nonumber
    \end{align}
    where $\theta_i=\arccos(l(\sigma_i^2))$. Note that
    \begin{equation*}
        \Delta_{\min} \leq {\min}\{2\pi-2\alpha,\alpha-\beta,2\beta\}.
    \end{equation*}
    It then follows from \Cref{thm:pointwise convergence} and \Cref{thm:convergence at dcon}
    that \eqref{Accuracy of projector1} holds.
    It is straightforward to justify from \eqref{Accuracy of projector1} that if
    $d$ satisfies \eqref{dsize} then $\|P_{S}-P\|<\frac{1}{4}$.

    It is known from \Cref{thm:Nonnegative} that the eigenvalues
    $\gamma_i=\psi_d(l(\sigma_i^2)),\,i=1,2,\ldots,n$ of $P$ are in $[0,1]$,
    showing that $P$ is SPSD. Therefore, from \eqref{diffeig} we
    obtain
    \begin{equation*}
        \|P_{S}-P\|=\max\biggl\{\max_{\sigma_{i}\in (a,b)}1-\gamma_i,
        \max_{\sigma_{i}= a \ {\rm or }\ b}\left|\frac{1}{2}-\gamma_i\right|,
        \max_{i = n_{sv}+1,\ldots,n}\gamma_i\biggr\}.
    \end{equation*}
    The above relation and \eqref{errorp} show that
    \begin{align*}
        0\leq 1-\gamma_i&<\frac{1}{4},\ \sigma_{i}\in (a,b),\\
        \left|\frac{1}{2}-\gamma_i\right|&<\frac{1}{4},\ \sigma_{i}= a \ {\rm or}\ b,\\
        0\leq \gamma_i&<\frac{1}{4},\ i=n_{sv}+1,\ldots,n.
    \end{align*}
    With the labeling order of $\gamma_i, i=1,2,\ldots,n$,
    the above proves \eqref{evhatp}. \qed
\end{proof}

\begin{remark}
    \Cref{thm:accuracyps1} shows that if the approximate spectral projector has some
    accuracy, e.g., \eqref{errorp}, then the dominant eigenvalues
    $\gamma_1,\ldots,\gamma_{n_{sv}}$ of $P$
    correspond to the desired singular values $\sigma_1,\ldots,\sigma_{n_{sv}}$
    and the associated dominant subspace are the corresponding
    right singular subspace.
    Moreover, if none of $a$ and $b$ is a singular value of $A$, then
    $\|P-P_S\|<\frac{1}{2}$ is enough to guarantee such properties.
    The previous example has justified that
    $\|P-P_S\|$ is reasonably small for a modest $d$; see \Cref{fig:bound}.
    In applications, we know nothing about the singular values of $A$ and $\Delta_{\min}$,
    and a practical selection strategy for $d$ is particularly appealing.
    Without a priori information on the distribution of singular values of $A$,
    suppose that the $\theta_i$ are uniformly distributed approximately, i.e.,
    $\Delta_{\min}\approx\frac{\alpha-\beta}{n_{sv}}$.
    Then \eqref{dsize} reads as
    \begin{equation*}
        d \geq \frac{\sqrt[3]{2}\pi^2 n_{sv}^{4/3}}{(\alpha-\beta)^{4/3}}-2.
    \end{equation*}
    However, the bounds in \Cref{thm:pointwise convergence}
    and \Cref{thm:convergence at dcon}, though the asymptotic
    convergence rates are optimal, are generally considerable overestimates,
    as \Cref{fig:bound} has indicated.
    A key is that the factor $\alpha-\beta$ in the denominator that is
    critical and determines the accuracy of $P$; the smaller $\alpha-\beta$ is,
    the harder it is to approximate the step function.
    Therefore, we propose to choose
    \begin{equation}\label{dchoice}
        d=\left\lceil \frac{D\pi^2}{(\alpha-\beta)^{4/3}}\right\rceil-2
    \end{equation}
    with $D$ some modest constant.
    We will propose selection strategies for choosing $D$ in \eqref{dchoice} in subsequent algorithms.
\end{remark}

\begin{remark}\label{rem:error}
    As $d$ increases, $\gamma_i\approx 1,\ i=1,2,\ldots,r$,
    $\gamma_i\approx \frac{1}{2},\ i=r+1,\ldots, n_{sv}$,
    and $\gamma_i\approx 0,\ i = n_{sv}+1,\ldots,n$.
    In fact, by \eqref{Accuracy of projector1},
    we can make $\|P_{S}-P\|<\epsilon$ with
    $\epsilon$ arbitrarily small by increasing $d$. In this case, we have
    \begin{align}
        1-\epsilon<&\gamma_i\leq 1, i = 1,2,\dots,r,\label{est1}\\
        \frac{1}{2}-\epsilon<&\gamma_i<\frac{1}{2}+\epsilon, i = r+1,\dots,n_{sv}, \label{est2}\\
        0\leq&\gamma_i<\epsilon,\ i = n_{sv}+1,\ldots,n. \label{est3}
    \end{align}
\end{remark}

\subsection{Estimates for the number of desired singular values}
Note that the trace ${\rm tr}(P_S)=r+\frac{n_{sv}-r}{2}=\frac{r+n_{sv}}{2}$,
which equals $n_{sv}$ when none of $a$ and $b$ is a singular value of $A$.
As \Cref{alg:subspace iteration} requires that the subspace dimension
$p\geq n_{nv}$, it is critical to reliably estimate $n_{sv}$. To this end,
we first show how to choose $d$ to ensure that ${\rm tr}(P)$ approximates
${\rm tr}(P_S)$ with an arbitrarily prescribed accuracy, and then making
use of \Cref{lem:stochastic estimation} to choose $p$ that ensures
$p \geq n_{sv}$ reliably.

\begin{theorem}\label{thm:trest}
    The trace ${\rm tr}(P)$ satisfies
    \begin{equation}\label{trest}
       |{\rm tr}(P_{S})-{\rm tr}(P)|\leq n\|P_{S}-P\| \leq \frac{n\pi^6}{2(d+2)^3\Delta_{\min}^4}
    \end{equation}
with $\Delta_{\min}$ defined by \eqref{dmin}.
\end{theorem}

\begin{proof}
We have
    \begin{align}
        |{\rm tr}(P_{S})- {\rm tr}(P)| &=\lvert\sum_{i=1}^n (f(\sigma_i^2)-\gamma_i)
        \rvert\leq \sum_{i=1}^n|f(\sigma_i^2)-\gamma_i| \label{over1}\\
                &\leq n \max_{i=1,2,\ldots,n}|f(\sigma_i^2)-\gamma_i| \label{over2}\\
                & = n\|P_{S}-P\|, \nonumber
    \end{align}
    which, together with \eqref{Accuracy of projector1}, proves \eqref{trest}. \qed
\end{proof}

\begin{remark}\label{rem:dchoice for nsv}
    Bound \eqref{trest} is generally
    very conservative since bounds \eqref{over1} and \eqref{over2}
    may be considerable overestimates by noticing that the signs of
    $
        f(\sigma_i^2)-\gamma_i=1-\gamma_i\geq 0,\ i=1,2,\ldots,r
    $
    and $ f(\sigma_i^2)-\gamma_i=-\gamma_i
    \leq 0, \ i=n_{sv}+1,\ldots,n$
    are opposite, and their sizes may differ greatly.
    Consequently, the factor $n$ typically behaves like $\mathcal{O}(1)$,
    so that, in terms of \Cref{thm:pointwise convergence}
    and \Cref{thm:convergence at dcon}, a modestly sized $d$ can ensure that
    the actual error is reasonably small.
\end{remark}

\begin{remark}\label{rem:Mchoice}
    Since $P$ is SPSD, we can exploit \Cref{lem:stochastic estimation} to
    derive a reliable estimate of
    ${\rm tr}(P)$ and use it as an approximation to ${\rm tr}(P_S)$.
    \Cref{lem:stochastic estimation} indicates that the smallest sample number
    $M\approx\frac{8\ln\frac{2}{\delta}}{\epsilon^2n_{sv}}$.
    Note that $\epsilon\in [10^{-2},10^{-1}]$
    means that $H_{M}$ is a reliable estimate for
    ${\rm tr}(P)$ with high probability $1-\delta\approx 1$ for $\delta \sim 10^{-2}$.
    For $n_{sv}$ ranging from a few to hundreds,
    a modest $M$ generally gives a reliable estimate for ${\rm tr}(P)$. Strikingly,
    for given $\epsilon$ and $\delta$, the bigger $n_{sv}$, the smaller $M$, i.e., the more
    easily it is to estimate a bigger $n_{sv}$.
\end{remark}

In summary, combining \Cref{rem:dchoice for nsv} and \Cref{rem:Mchoice},
we conclude that $H_M$ is a reliable estimate for ${\rm tr}(P_{S})$
when $M$ and $d$ are modest. Numerical experiments
in \Cref{sec: experiments} will show that taking
$d$ for $D \in [2, 10]$ in \eqref{dchoice} is reliable and produces
almost unchanged $H_M$'s. We present \Cref{alg:subspace dimension1} to estimate $n_{sv}$,
where $P$ is not formed explicitly and $H_M$ is efficiently computed by
exploiting the three term recurrence of Chebyshev polynomials.
In this way, it is, though a little tedious, easy to verify that the
computation of $H_{M}$ totally requires $2Md$ MVs and approximately $6Mnd$ flops,
where MV denotes a matrix-vector product with $A$ or $A^T$.

\begin{algorithm}
    \caption{Estimation of the number $n_{sv}$}
    \label{alg:subspace dimension1}
    \begin{algorithmic}[1]
        \REQUIRE{The matrix $A$, the interval $[a, b]$, the series degree $d$, and $M$ Rademacher random $n$-vectors $z_1,z_2,\dots,z_M$.}
        \ENSURE{Take $H_{M}$ as an estimate for $n_{sv}$.}
        \STATE{Apply \Cref{alg:Spectral Projector1} to compute
        the Chebyshev--Jackson coefficients.}
        \STATE{Compute $H_{M}=\frac{1}{M}\sum_{i=1}^{M}z_i^TPz_i
        =\frac{1}{M}\sum_{i=1}^{M}\sum_{j=0}^d \rho_{j,d}c_j z_i^T T_j(l(S))z_i$.}
    \end{algorithmic}
\end{algorithm}

With $H_M$ available, we find that taking
\begin{equation}\label{pchoice}
    p=\lceil \mu H_{M}\rceil
\end{equation}
with $\mu\geq 1.1$ can ensure the subspace dimension $p\geq n_{sv}$,
where $\lceil \cdot \rceil$ is the ceil function.
In fact, \Cref{lem:stochastic estimation} shows
that $|H_{M}-{\rm tr}(P)|\leq \epsilon \ {\rm tr}(P)$
with the high probability $1-\delta\approx 1$ for a modest $M$.
Therefore, $H_{M} \geq (1-\epsilon){\rm tr}(P)$ and $\mu H_{M} \geq \mu(1-\epsilon){\rm tr}(P)$.
Obviously, $\mu=1.1$ ensures that $\mu(1-\epsilon)\geq 1$ with $\epsilon \leq \frac{1}{11}$.
As a result, $p$ in \eqref{pchoice} is a reliable upper bound
for ${\rm tr}(P)$ with high probability when $M$ is of modest size.
On the other hand, ${\rm tr}(P)$
is a good approximation to ${\rm tr}(P_S)$ for a proper series
degree $d$. Therefore, once $M$ and $d$ are suitably chosen, $p$ in \eqref{pchoice}
can ensure $p\geq n_{sv}$ with high probability.
However, different $p$'s may affect the overall efficiency of the
CJ-FEAST SVDsolver. We will come back to the choice of $\mu$ after we establish
the convergence of the CJ-FEAST SVDsolver.

\subsection{The algorithm and some details}

Having determined the approximate spectral projector $P$ and the subspace dimension $p\geq n_{sv}$,
we apply \Cref{alg:subspace iteration} to $P$, and form an approximate
eigenspace of $P$ associated with its $p$ dominant eigenvalues
$\gamma_i,\ i=1,2,\ldots,p$.
We then take the current subspace as the
right projection subspace $\mathcal{V}^{(k)}$, form
the left projection subspace $\mathcal{U}^{(k)}=A\mathcal{V}^{(k)}$,
and project $A$ onto them to compute Ritz approximations
$(\hat{\sigma}_i^{(k)},\hat{u}_i^{(k)},\hat{v}_i^{(k)})$ of
the desired $n_{sv}$ singular triplets
$(\sigma_i,u_i,v_i)$, $i=1,2,\ldots,n_{sv}$.
Precisely, let the columns of $Q_1^{(k)}\in \mathbb{R}^{n\times p}$
form an orthogonal basis of $\mathcal{V}^{(k)}$ and
$AQ_1^{(k)}=Q_2^{(k)}\bar{A}^{(k)}$
be the thin QR factorizations of $AQ_1^{(k)}$, where $\bar{A}^{(k)}\in \mathbb{R}^{p\times p}$
is upper triangular. Then the columns of $Q_2^{(k)}$ form orthonormal basis
of $\mathcal{U}^{(k)}=A\mathcal{V}^{(k)}$, and
$(Q_2^{(k)})^TAQ_1^{(k)}=\bar{A}^{(k)}$
is the projection matrix.
We describe the procedure as \Cref{alg:Crossproduct-PSVD}.
The computational cost of one iteration of \Cref{alg:Crossproduct-PSVD}
is listed in \Cref{tab:Computational cost}, where, at Step 7,
we exploit the fact that the upper part of the residual of
$(\hat{\sigma}_i^{(k)}, \hat{u}_i^{(k)},\hat{v}_i^{(k)})$ is zero
and we do not compute it.

\begin{algorithm}
    \caption{The CJ-FEAST SVDsolver}
    \label{alg:Crossproduct-PSVD}
    \begin{algorithmic}[1]
        \REQUIRE{The matrix $A$, the interval $[a,b]$, the series degree $d$, and an $n$-by-$p$ column orthonormal matrix $\hat{V}^{(0)}$ with $p\geq n_{sv}$.}
        \ENSURE{The $n_{sv}$ converged Ritz triplets  $(\hat{\sigma}_i^{(k)}, \hat{u}_i^{(k)},\hat{v}_i^{(k)})$ with $\hat{\sigma}_i^{(k)}\in [a,b]$.}
        \STATE{Apply \Cref{alg:Spectral Projector1} to compute the Chebyshev--Jackson coefficients.}
        \FOR{$k=1,2,\dots,$}
        \STATE{Compute $ Y^{(k)}=P\hat{V}^{(k-1)}=\sum_{j=0}^d \rho_{j,d}c_j T_j(l(S))\hat{V}^{(k-1)}$.}
        \STATE{Make QR factorizations, and compute the projection matrix $\bar{A}^{(k)}$:\\ $Y^{(k)}=Q_{1}^{(k)}R_{1}^{(k)}$ and $AQ_{1}^{(k)}=Q_{2}^{(k)}\bar{A}^{(k)}$.}
        \STATE{Compute the SVD:    $\bar{A}^{(k)}=\bar{U}^{(k)}\hat\Sigma^{(k)}(\bar{V}^{(k)})^T$
        with $\hat{\Sigma}^{(k)}={\rm diag}(\hat{\sigma}_1^{(k)},\ldots,\hat{\sigma}_p^{(k)})$.}
        \STATE{Form the approximate left and right singular vector matrices
        $\hat{U}^{(k)}=Q_{2}^{(k)}\bar{U}^{(k)}$ and $ \hat{V}^{(k)}=Q_{1}^{(k)}\bar{V}^{(k)}$.}
        \STATE{Pick up $\hat{\sigma}_i^{(k)}\in [a,b]$, compute the residual norms of the Ritz approximations
        $(\hat{\sigma}_i^{(k)}, \hat{u}_i^{(k)},\hat{v}_i^{(k)})$,
        where $\hat{u}_i^{(k)}=\hat{U}^{(k)}e_i$ and $\hat{v}_i^{(k)}
        =\hat{V}^{(k)}e_i$, and test convergence.}
        \ENDFOR
    \end{algorithmic}
\end{algorithm}

\begin{table}[h]
        \caption{Computational cost of one iteration of \Cref{alg:Crossproduct-PSVD}.}
    \label{tab:Computational cost}
        \begin{tabular}{|c|c|c|}
            \hline
            Steps      & MVs      & flops                    \\
            \hline
            3          & $2dp$  & $4npd$                   \\
            4          &  $p$        & $2(m+n)p^2$              \\
            5          &          & $21p^3$                  \\
            6          &          & $2(m+n)p^2$              \\
            7         & $p$      & $2np$                    \\
            \hline
            Total cost & $2(d+1)p$ & $4ndp+4(m+n)p^2+2np+21p^3$ \\
            \hline
        \end{tabular}
\end{table}

Suppose that $A$ is sparse and has $\mathcal{O}(m+n)$ nonzero entries, and
take the subspace dimension
$p=\mathcal{O}(n_{sv})$ with the constant in $\mathcal{O}(\cdot)$
comparable to but bigger than one. Then from the table we see that
$2(d+1)p$ MVs cost $\mathcal{O}(2(m+n)dn_{sv})$ flops and
$4ndp+4(m+n)p^2+2np+21p^3=\mathcal{O}(ndn_{sv})+\mathcal{O}((m+n)n_{sv}^2)$.
Therefore, the flops of MVs is comparable to
the other cost when $d \geq \mathcal{O}(n_{sv})$.
If $A$ is not sparse
and non-structured, i.e., the number of its nonzero entries
is $\mathcal{O}(mn)$, then MVs cost $\mathcal{O}(2mndp)$
flops and overwhelm the others unconditionally.
As a result, we can measure the overall efficiency
of \Cref{alg:Crossproduct-PSVD} by MVs.

\section{Convergence of the CJ-FEAST SVDsolver}\label{conver}

This section is devoted to a convergence analysis of
\Cref{alg:Crossproduct-PSVD}.
We will establish several convergence results on the solver.

Recall from \eqref{ceigen} that the columns of $V$
are the right singular vectors of $A$. We partition
$V=[V_p,V_{p,\perp}]$, and set up the following notation:
\begin{align}
    V_p        & =[v_1,\dots, v_p],  \ \ V_{p,\perp}=[v_{p+1},\dots, v_{n}],\label{vdef}                                      \\
    \Gamma_{p} & =\diag(\gamma_1, \dots, \gamma_p), \ \ \Gamma_p^{\prime}=\diag(\gamma_{p+1}, \dots, \gamma_{n}),\label{gammadef} \\
    \Sigma_{p} & =\diag(\sigma_1, \dots,\sigma_p), \ \ \Sigma_p^{\prime}=\diag(\sigma_{p+1}, \dots, \sigma_{n}). \label{sigmadef}
\end{align}

It is easy to see that \Cref{alg:Crossproduct-PSVD} generates the subspaces
\begin{equation*}
    {\rm span}\{\hat{V}^{(k)}\}={\rm span}\{Q_1^{(k)}\}={\rm span}\{Y^{(k)}\}=P{\rm span}\{\hat{V}^{(k-1)}\},
\end{equation*}
showing that
\begin{equation}\label{subspace recursion}
    {\rm span}\{\hat{V}^{(k)}\}=P^k{\rm span}\{\hat{V}^{(0)}\}.
\end{equation}

\begin{theorem}\label{thm:subspace convergence}
    Suppose that $V_p^T\hat{V}^{(0)}$ is invertible and $\gamma_p>\gamma_{p+1}$. Then
    \begin{equation}\label{q1k}
        Q_{1}^{(k)}=(V_p+V_{p,\perp}E^{(k)})(M^{(k)})^{-\frac{1}{2}}U^{(k)}
    \end{equation}
    with
    \begin{align}
         & E^{(0)}=V_{p,\perp}^T\hat{V}^{(0)}(V_p^T\hat{V}^{(0)})^{-1},\
         E^{(k)}=\Gamma_p^{\prime k}E^{(0)}\Gamma_{p}^{-k},\label{Ek} \\
         & M^{(k)}=I+(E^{(k)})^TE^{(k)}\label{mkdef}
    \end{align}
    and $U^{(k)}$ being an orthogonal matrix; furthermore,
    \begin{equation}\label{normek}
        \|E^{(k)}\|\leq \biggl(\frac{\gamma_{p+1}}{\gamma_p}\biggr)^k\|E^{(0)}\|,
    \end{equation}
    and the distance $\epsilon_k:={\rm dist}({\rm span}\{Q_1^{(k)}\},{\rm span}\{V_p\})$ between
${\rm span}\{Q_1^{(k)}\}$ and ${\rm span}\{V_p\}$ {\rm (cf. \cite[Section 2.5.3]{golub2013matrix})} satisfies
    \begin{equation}\label{dist}
      \epsilon_k= \frac{\|E^{(k)}\|}{\sqrt{1+\|E^{(k)}\|^2}}\leq \biggl(\frac{\gamma_{p+1}}{\gamma_p}\biggr)^k\|E^{(0)}\|.
    \end{equation}
     Label $\hat{\sigma}_{1}^{(k)},\dots, \hat{\sigma}_{p}^{(k)}$ in the same order as $\sigma_1,\dots, \sigma_p$ in
    \Cref{thm:accuracyps1}. Then
    \begin{equation}\label{sigma}
        |(\hat\sigma_{i}^{(k)})^2-\sigma_i^2| \leq \|A\|^2 (3\epsilon_k^2+\epsilon_k^4),\ i=1,2,\ldots,n_{sv}.
    \end{equation}
\end{theorem}

\begin{proof}
    Expand $\hat{V}^{(0)}$ as the orthogonal direct sum of $V_p$ and $V_{p,\perp}$. Then
    \begin{equation*}
        \hat{V}^{(0)}=V_pV_p^T\hat{V}^{(0)}+V_{p,\perp}V_{p,\perp}^T\hat{V}^{(0)}
        =(V_p+V_{p,\perp}V_{p,\perp}^T\hat{V}^{(0)}(V_p^T\hat{V}^{(0)})^{-1})
        V_p^T\hat{V}^{(0)}.
    \end{equation*}
    From this and the first relation in \eqref{Ek} it follows that
    \begin{equation*}
        \hat{V}^{(0)}(V_p^T\hat{V}^{(0)})^{-1}=V_p+V_{p,\perp}E^{(0)}.
    \end{equation*}
    By $PV_p=V_p\Gamma_p$ and $PV_{p,\perp}=V_{p,\perp}\Gamma_p^{\prime}$, we obtain
    $P^kV_p=V_p\Gamma_p^k$ and
    $P^kV_{p,\perp}=V_{p,\perp}\Gamma_p^{\prime k}$. Therefore,
    \begin{align}
        P^k\hat{V}^{(0)}(V_p^T\hat{V}^{(0)})^{-1}\Gamma_{p}^{-k}
        &=V_p+P^kV_{p,\perp}E^{(0)}\Gamma_{p}^{-k} \nonumber\\
        &=V_p+V_{p,\perp}\Gamma_p^{\prime k}E^{(0)}\Gamma_{p}^{-k}
        =V_p+V_{p,\perp}E^{(k)} \label{eqsub}
    \end{align}
    with $E^{(k)}$ defined by \eqref{Ek}. It is straightforward that
    \begin{equation*}
        \|E^{(k)}\|\leq \biggl(\frac{\gamma_{p+1}}{\gamma_p}\biggr)^k\|E^{(0)}\|,
    \end{equation*}
    which is \eqref{normek}.  By \eqref{eqsub}, we obtain
    \begin{equation*}
        {\rm span}\{Q_{1}^{(k)}\}=P^k{\rm span}\{\hat{V}^{(0)}\}={\rm span}\{V_p+V_{p,\perp}E^{(k)}\}.
    \end{equation*}
    Since $Q_{1}^{(k)}$ is column orthonormal, we can express $Q_{1}^{(k)}$ as
    \begin{equation*}
        Q_{1}^{(k)}=(V_p+V_{p,\perp}E^{(k)})(M^{(k)})^{-\frac{1}{2}}U^{(k)},
    \end{equation*}
    which establishes \eqref{q1k},
    where
    \begin{equation*}
        M^{(k)}=(V_p+V_{p,\perp}E^{(k)})^{T}(V_p+V_{p,\perp}E^{(k)})=I+(E^{(k)})^TE^{(k)}
    \end{equation*}
    is the matrix in \eqref{mkdef} and $U^{(k)}$ is an orthogonal matrix.

    By the distance definition \cite[Section 2.5.3]{golub2013matrix} of two subspaces, from \eqref{q1k} we have
    \begin{equation*}
        \epsilon_k = \|V_{p,\perp}^TQ_1^{(k)}\|=\|E^{(k)}(M^{(k)})^{-1/2}U^{(k)}\|=\frac{\|E_k\|}{\sqrt{1+\|E_k\|^2}},
    \end{equation*}
    which, together with \eqref{normek}, proves \eqref{dist}.

    Exploiting \eqref{ceigen} and \eqref{q1k}, we obtain
    \begin{align*}
         & \|U^{(k)}(Q_{1}^{(k)})^T S Q_{1}^{(k)}(U^{(k)})^T-\Sigma_p^2\|       \\
         & =\|(M^{(k)})^{-1/2}(V_p^T+(E^{(k)})^TV_{p,\perp}^T) V\Sigma^2V^T (V_p+V_{p,\perp}E^{(k)})(M^{(k)})^{-1/2}-\Sigma_p^2\|   \\
         & =\|(M^{(k)})^{-1/2}(\Sigma_p^2+(E^{(k)})^T(\Sigma_p^{\prime})^2E^{(k)})(M^{(k)})^{-1/2}-\Sigma_p^2\|    \\
         & \leq \|(M^{(k)})^{-1/2}\Sigma_p^2(M^{(k)})^{-1/2}-\Sigma_p^2\|+
         \|(M^{(k)})^{-1/2}(E^{(k)})^T(\Sigma_p^{\prime})^2E^{(k)}(M^{(k)})^{-1/2}\|.
    \end{align*}
    Let $F^{(k)}=I-(M^{(k)})^{-\frac{1}{2}}.$
    Then
    \begin{equation*}
        \|F^{(k)}\|=\|I-(M^{(k)})^{-\frac{1}{2}}\| = 1-\frac{1}{\sqrt{1+\|E^{(k)}\|^2}}\leq \frac{\|E^{(k)}\|^2}{1+\|E^{(k)}\|^2} =\epsilon_k^2.
    \end{equation*}
    Therefore,
    \begin{align*}
        &\|(M^{(k)})^{-1/2}\Sigma_p^2(M^{(k)})^{-1/2}-\Sigma_p^2\|=
        \|(I-F^{(k)})\Sigma_p^2(1-F^{(k)})-\Sigma_p^2\| \\
        &=\|-\Sigma_p^2F^{(k)}-F^{(k)}\Sigma_p^2+F^{(k)}\Sigma_p^2F^{(k)}\|
        \leq \|\Sigma_p^2\|(2\epsilon_k^2+\epsilon_k^4),
    \end{align*}
    which, together with
    \begin{equation*}
        \|(M^{(k)})^{-1/2}(E^{(k)})^T(\Sigma_p^{\prime})^2E^{(k)}(M^{(k)})^{-1/2}\| \leq \|A\|^2 \epsilon_k^2,
    \end{equation*}
    yields
    \begin{equation*}
        \|U^{(k)}(Q_{1}^{(k)})^T S Q_{1}^{(k)}(U^{(k)})^T-\Sigma_p^2\| \leq \|A\|^2 (3\epsilon_k^2+\epsilon_k^4).
    \end{equation*}
    Since the eigenvalues of $U^{(k)}(Q_{1}^{(k)})^T S Q_{1}^{(k)}(U^{(k)})^T$
    are $(\hat\sigma_{i}^{(k)})^2,\ i=1,2,\ldots,p$,
    by a standard perturbation
    result \cite[Corollary 8.1.6]{golub2013matrix},
    the above relation leads to \eqref{sigma}. \qed
\end{proof}

The following theorem establishes convergence results on the
Ritz vectors $\hat{u}_i^{(k)}$ and $\hat{v}_i^{(k)}$ and a new
convergence result on the Ritz values $\hat{\sigma}_i^{(k)}$.

\begin{theorem}\label{thm:triplets convergence}
    Let $\beta^{(k)}=\|P^{(k)}S(I-P^{(k)})\|$,
    where $P^{(k)}$ is the orthogonal projector onto \textnormal{span}$\{Q_1^{(k)}\}$.
    Assume that each singular value of $A$ in $[a,b]$ is simple, and define
    \begin{equation}\label{gap}
        \delta_i^{(k)}= \underset{j\neq i}{\min}
        |\sigma_i^2-(\hat{\sigma}_j^{(k)})^2|,\ i=1,2,\ldots,n_{sv}.
    \end{equation}
    Then for $i=1,2,\dots,n_{sv}$ it holds that
    \begin{align}
        \sin\angle(\hat{v}_{i}^{(k)},v_i) & \leq \sqrt{1+\frac{(\beta^{(k)})^2}{(\delta_i^{(k)})^2}} \biggl(\frac{\gamma_{p+1}}{\gamma_i}\biggr)^k\|E^{(0)}\|,\label{sinev}     \\
        \sin\angle(\hat{u}_{i}^{(k)},u_i) & \leq  \frac{\|A\|}{\hat\sigma_{i}^{(k)}} \sin\angle(\hat{v}_{i}^{(k)},v_i),    \label{sineu}    \\
        |(\hat\sigma_{i}^{(k)})^2-\sigma_i^2|   & \leq \|A\|^2\sin^2\angle(\hat{v}_{i}^{(k)},v_i).\label{sigmacondition}
    \end{align}
\end{theorem}

\begin{proof}
    Note that $((\hat{\sigma}_i^{(k)})^2,\hat{v}_{i}^{(k)})$,
    $1\leq i\leq n_{sv}$ are the Ritz pairs of $S$ with respect
    to ${\rm span}\{Q_1^{(k)}\}$. Applying \cite[Theorem 4.6, Proposition 4.5]{saad2011numerical} to our case yields
    \begin{align}
        \sin\angle(\hat{v}_{i}^{(k)},v_i)     & \leq \sqrt{1+\frac{(\beta^{(k)})^2}{(\delta_i^{(k)})^2}} \sin\angle(v_i, {\rm span}\{Q_1^{(k)}\}),\label{ab}      \\
        |(\hat{\sigma}_i^{(k)})^2-\sigma_i^2| & \leq \|S-\sigma_i^2I\|\sin^2\angle(\hat{v}_{i}^{(k)},v_i) \leq \|A\|^2\sin^2\angle(\hat{v}_{i}^{(k)},v_i),\nonumber
    \end{align}
    which proves \eqref{sigmacondition}.

    From \eqref{q1k} and \eqref{gammadef}, we obtain
    \begin{align*}
        \sin\angle(v_i,{\rm span}\{Q_1^{(k)}\}) &=\sin\angle(v_i,{\rm span}\{Q_1^{(k)}(U^{(k)})^T(M^{(k)})^{1/2}\})\\
        &= \sin\angle(v_i,V_p+V_{p,\perp}E^{(k)}) \leq
        \sin\angle(v_i,v_i+V_{p,\perp}E^{(k)}e_i)                                                                                  \\
        & =\frac{\|E^{(k)}e_i\|}{\sqrt{1+\|E^{(k)}e_i\|^2}}\leq \|E^{(k)}e_i\| \\
        &=\|\Gamma_p^{\prime k} E^{(0)}\Gamma_{p}^{-k}e_i\|
        \leq\|\Gamma_p^{\prime k} E^{(0)}\|\gamma_{i}^{-k}                                                                         \\
        & \leq \biggl(\frac{\gamma_{p+1}}{\gamma_i}\biggr)^k \|E^{(0)}\|.
    \end{align*}
    In terms of the notation in Steps 3--5 of \Cref{alg:Crossproduct-PSVD}, we have
    \begin{equation*}
        \hat{U}^{(k)}\hat{\Sigma}^{(k)}=Q_2^{(k)}\bar{U}^{(k)}\hat{\Sigma}^{(k)} = Q_2^{(k)}\bar{A}^{(k)}\bar{V}^{(k)}
        =AQ_1^{(k)}\bar{V}^{(k)}=A\hat{V}^{(k)},
    \end{equation*}
    showing that
    \begin{equation}\label{leftright}
        A\hat{v}_{i}^{(k)}=\hat{\sigma}_i^{(k)}\hat{u}_{i}^{(k)}.
    \end{equation}

    Decompose $\hat{v}_{i}^{(k)}$ into the orthogonal direct sum:
    \begin{equation*}
        \hat{v}_{i}^{(k)}=v_i\cos\angle(\hat{v}_{i}^{(k)},v_i)+
        z\sin\angle(\hat{v}_{i}^{(k)},v_i),
    \end{equation*}
    where $z$ is orthogonal to $v_i$ with $\|z\|=1$. Abbreviate $\angle(\hat{v}_{i}^{(k)},v_i)$ as $\phi_i$. Then
    \begin{equation}\label{uhat}
        \hat\sigma_{i}^{(k)}\hat{u}_{i}^{(k)}=
        A\hat{v}_{i}^{(k)}=A(v_i\cos\phi_i+z\sin\phi_i)
        =\sigma_iu_i\cos\phi_i+Az\sin\phi_i.
    \end{equation}
    Since
    \begin{equation*}
        u_i^TAz=z^TA^Tu_i=\sigma_iz^Tv_i=0,
    \end{equation*}
    it follows from \eqref{uhat} that
    \begin{equation*}
        \sin\angle(\hat{u}_{i}^{(k)},u_i)  =  \frac{\|Az\|}{\hat\sigma_{i}^{(k)}}\sin\phi_i \leq \frac{\|A\|}{\hat\sigma_{i}^{(k)}}\sin\phi_i,
    \end{equation*}
    which proves \eqref{sineu}. \qed
\end{proof}

This theorem indicates that, provided that $\delta_i^{(k)}$ defined by
\eqref{gap} is uniformly bounded from below with respect to iteration
$k$, the left and right Ritz vectors
$\hat{u}_i^{(k)}$ and $\hat{v}_i^{(k)}$ converge at least with the linear
convergence factor $\frac{\gamma_{p+1}}{\gamma_i}$
but the Ritz value $\hat{\sigma}_i^{(k)}$ converges at least with the factor
$(\frac{\gamma_{p+1}}{\gamma_i})^2$.
This indicates that the errors of the Ritz values are roughly the squares
of those of the corresponding left and right Ritz vectors.

\begin{remark}
    The slowest convergence factor $\frac{\gamma_{p+1}}{\gamma_{n_{sv}}}$ is
    affected by the series degree $d$ and the subspace dimension $p$.
    Increasing $d$ or $p$ will make this factor smaller, but will consume more
    computational cost in one iteration (cf. \Cref{tab:Computational cost}).
    For a modestly sized $d$, increasing $p$ will reduce the number of
    iterations;
    for $d$ very large, increasing $p$ does not
    reduce the number of iterations essentially
    since, for a very good approximate spectral projector $P$,
    the solver will converge in very few iterations.
    Numerical experiments in \Cref{sec: experiments} will illustrate
    that choosing $d$ as \eqref{dchoice} with $D\in [2,10]$ and $p$ as
    \eqref{pchoice} with $\mu \in [1.1, 1.5]$ is reliable and works well.
\end{remark}

\section{Numerical experiments}\label{sec: experiments}
We now report numerical experiments, and
provide a detailed numerical justification of \Cref{alg:subspace dimension1}
and \Cref{alg:Crossproduct-PSVD}, the theoretical results and remarks. The
test matrices are from \cite{davis2011university}, and we list some of their
basic properties and the interval $[a,b]$ of interest
in \Cref{tab:Properties of test matrices}. As we see,
the matrices $A$ range from
rank deficient to well conditioned, and the locations of intervals and the
widths relative to the whole singular spectra differ considerably. We will
also find that the numbers $n_{sv}$'s of the desired singular triplets
differ greatly too. Therefore, our concerning SVD problems
are representative in applications, implying that our test results and assertions are of generality.

In the experiments, an approximate singular triplet $(\hat{\sigma}, \hat{u}, \hat{v})$
is claimed to have converged if the residual norm satisfies
\begin{equation}\label{stopcriterion}
    \|r(\hat{\sigma}, \hat{u}, \hat{v})\| \leq \|A\| \cdot tol.
\end{equation}
We will use $tol=1e-8$ and $1e-12$ to test first ten examples and $tol=1e-13$ to test the last example.

All the numerical experiments were performed on an Intel Core i7-9700,
CPU 3.0GHz, 8GB RAM using MATLAB R2022a with the machine precision
$\epsilon_{{\rm mach}}=2.22e-16$  under the Microsoft Windows 10 64-bit system.
To make a fair comparison, for each test problem and given the subspace dimension $p$, we used the same starting $n\times p$ orthonormal
$\hat{V}^{(0)}$ in all the algorithms, which is obtained by the thin
QR decomposition of a random matrix generated in
a normal distribution.

\begin{table}[h]
    \caption{Properties of test matrices, where the
        $nnz(A)$ is the number of nonzero
        entries in $A$, and the largest and smallest
        singular values $\|A\|$ and
        $\sigma_{\min}(A)$ of $A$ are from \cite{davis2011university}.}
        \label{tab:Properties of test matrices}
    \centering
        \resizebox*{\textwidth}{!}{
            \begin{tabular}{|c|c|c|c|c|c|c|}
                \hline
                Matrix $A$ & $m$    & $n$   & $nnz(A)$ & $\|A\|$ & $\sigma_{\min}(A)$ & $[a,b]$ \\
                \hline
                GL7d12          & 8899   & 1019  & 37519    & 14.4 & 0                  & $[11, 12]$           \\
                plat1919        & 1919   & 1919  & 32399    & 2.93 & 0                  & $[2.1, 2.5]$         \\
                flower\_5\_4    & 5226 & 14721 & 43942  & 5.53 & $3.70e-1$         & $[4.1,4.3]$ \\
                fv1             & 9604   & 9604  & 85264    & 4.52 & $5.12e-1$           & $[3.1, 3.15]$         \\
                3elt\_dual      & 9000  & 9000    & 26556   & 3.00 & $6.31e-13$          & $[1.5, 1.6]$ \\
                rel8            & 345688 & 12347 & 821839   & 18.3 & 0                  & $[13, 14]$           \\
                crack\_dual     & 20141 & 20141  & 60086    & 3.00 & $1.73e-4$           & $[1, 1.1]$ \\
                nopoly          & 10774 & 10774  & 	70842  & 23.3 & $1.91e-15$          & $[12, 12.5]$ \\
                barth5          & 15606 & 15606  & 61484    & 4.23 & $7.22e-11$         & $[1.5, 1.6]$ \\
                L-9             & 17983 & 17983     & 71192 & 4.00  & 0                 & $[1.2 ,1.3]$ \\
                shuttle\_eddy   & 10429 & 10429  & 103599   & 16.2  & 0                 & $[7, 7.01]$ \\
                \hline
            \end{tabular}
        }
\end{table}

\subsection{Estimations of the number of desired singular values}\label{sub6.1}
We first justify that our
estimates for $n_{sv}$'s are reliable.
The exact singular values and $n_{sv}$'s are from \cite{davis2011university}.
For each test problem,
we take the polynomial degree $d$ in \eqref{dchoice} using $D=2,4,8$,
compute $H_M$ for two modestly sized $M=20,30$, and list them
in \Cref{tab:number estimation}.
We see that, for each problem, all the $H_{M}$ are accurate estimates
for $n_{sv}$, and they remain almost unchanged.
These results demonstrate that our selection strategy $D\in [2,10], M\in [20,30]$
is reliable. We suggest to use the smaller $M=20$ and the smallest $D=2$,
which cost the least.
Moreover, the numerical results indicates that the subspace dimension
$p=\lceil 1.1H_{M}\rceil \geq n_{sv}$,
which illustrates that our selection strategy \eqref{pchoice} with $\mu \geq 1.1$ is
reliable to guarantee that $p\geq n_{sv}$ in computations.
\begin{table}
    \caption{The exact $n_{sv}$ and their estimates $H_{M}$.}
    \label{tab:number estimation}
    \begin{tabular}{|c|c|c|c|c|c|}
        \hline
        \multirow{2}{*}{Matrix}     & \multirow{2}{*}{$n_{sv}$}  & \multirow{2}{*}{$M$} &  \multicolumn{3}{c|}{$H_{M}$} \\
                    \cline{4-6}
                                    &                            &                        &   $D=2$ & $D=4$ & $D=8$ \\
        \hline
        \multirow{2}{*}{GL7d12}    & \multirow{2}{*}{17}     & $20$   & 18.2    & 18.1  & 17.5        \\
                                    &                       & $30$   & 16.9    & 17.6  & 18.5            \\
        \hline
        \multirow{2}{*}{plat1919}  & \multirow{2}{*}{8}      & $20$   & 7.2    & 7.8  & 8.0        \\
                                    &                       & $30$    & 9.2    & 8.4  & 8.4            \\
        \hline
        \multirow{2}{*}{flower\_5\_4}     & \multirow{2}{*}{137}    & $20$   & 129.3  & 127.3  &  131.4       \\
                                    &                               & $30$   & 131.0  & 133.4  &  135.4           \\
        \hline
        \multirow{2}{*}{fv1}        & \multirow{2}{*}{89}     & $20$   & 93.4    & 93.8  & 92.1        \\
                                    &                         & $30$   & 90.8    & 91.8  & 89.4            \\
        \hline
        \multirow{2}{*}{3elt\_dual}     & \multirow{2}{*}{368}    & $20$   & 360.4  & 354.3  &  374.5       \\
                                    &                             & $30$   & 368.4  & 370.7  &  370.1           \\
        \hline
        \multirow{2}{*}{rel8}   & \multirow{2}{*}{13}     & $20$   & 11.8    & 13.5  & 12.7        \\
                                &                         & $30$   & 14.1    & 11.8  & 12.7           \\
        \hline
        \multirow{2}{*}{crack\_dual}     & \multirow{2}{*}{330}    & $20$   & 333.2  & 331.0 &  329.0       \\
                                    &                             & $30$   & 335.7  & 333.8  &  330.6          \\
        \hline
        \multirow{2}{*}{nopoly}     & \multirow{2}{*}{340}    & $20$   & 335.3  & 336.1 & 347.3        \\
                                    &                             & $30$   & 345.2  & 337.7  & 337.9           \\
        \hline
        \multirow{2}{*}{barth5}     & \multirow{2}{*}{384}    & $20$   & 373.7  & 382.0 & 372.1        \\
                                    &                             & $30$   & 388.0  & 382.6  & 380.9           \\
        \hline
        \multirow{2}{*}{L-9}     & \multirow{2}{*}{477}    & $20$   &  486.2 & 483.3 &  480.9       \\
                                    &                             & $30$   & 479.8  & 484.6  &  481.4          \\
        \hline
        \multirow{2}{*}{shuttle\_eddy}     & \multirow{2}{*}{6}    & $20$   & 5.6  & 5.7 & 6.4        \\
                                    &                             & $30$   & 6.7  & 6.1 & 7.3           \\
        \hline
    \end{tabular}
\end{table}

\subsection{The case that the subspace dimension is smaller than the number of desired singular values}
Our theoretical results and analysis imply that \Cref{alg:Crossproduct-PSVD} with $p < n_{sv}$
should not work generally since we may
have $\gamma_i, i=1,2,\dots, p+1$ are almost equal.
As a result, subspace iteration either converges extremely slowly or
fails to converge.
To numerically justify these predictions,
we take $d=d_0$, the smallest integer that satisfies \eqref{dsize},
and $p<n_{sv}$, apply \Cref{alg:Crossproduct-PSVD} to
the test matrices rel8 and plat1919, and investigate the convergence behavior.

For rel8, we first take $p=n_{sv}=13$.
We observe that \Cref{alg:Crossproduct-PSVD} converges very fast and
all the thirteen desired singular triplets have been found when $k=2$.
But for $p=12<n_{sv}$, the residual norms of some of the Ritz triplets do not
decrease from the first iteration to $k=10$; in fact, the smallest relative
residual norms among the twelve ones stabilize around $3.43e-5$.

We have observed similar phenomena for plat1919.
For $p=n_{sv}=8$, all the eight Ritz triplets have converged when $k=2$.
But for $p=6<n_{sv}$, the algorithm fails, and the residual norms of
some Ritz triplets almost stagnate
from the first iteration to $k=20$,
and the smallest relative residual norms stabilize around $9.10e-3$.
\Cref{fig:notconverge} depicts the convergence processes of
the smallest relative residual norms for re18 and plat1919,
where the residual norms stagnate from the first iteration
onwards. Therefore, to make the algorithm work, one must take $p\geq n_{sv}$.

\begin{figure}[tbhp]
    \centering
    \subfloat[rel8]{\includegraphics[scale=0.4]{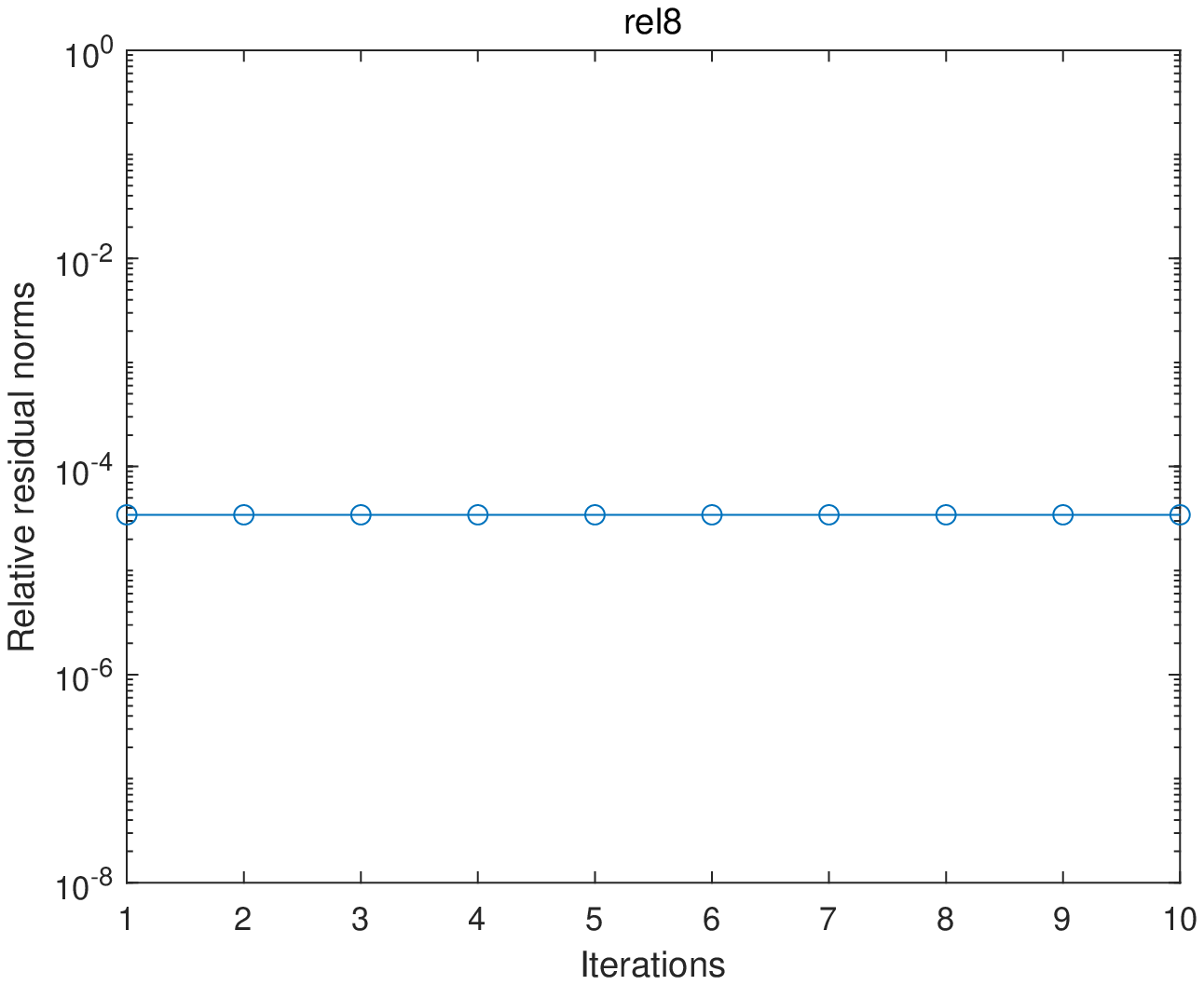}}
    \subfloat[plat1919]{\includegraphics[scale=0.4]{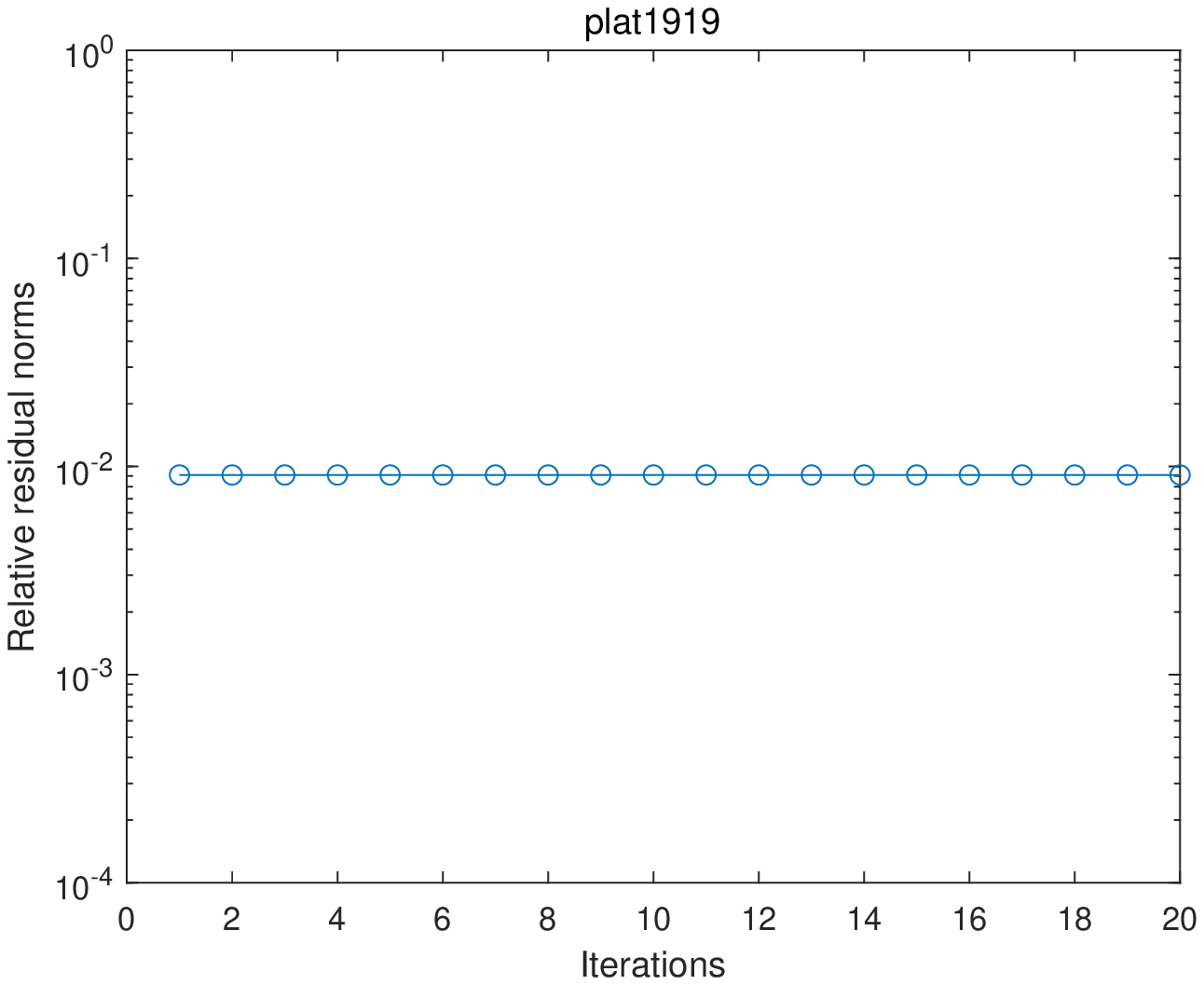}}
    \caption{Convergence processes when $p<n_{sv}$.}
    \label{fig:notconverge}
\end{figure}

Very importantly, our analysis and numerical
justification enable us to detect if $p\geq n_{sv}$ is met:
for a reasonably big $d$, if the algorithm converges extremely
slowly, then it is very possible that $p<n_{sv}$; we must stop the algorithm,
choose a bigger $p\geq n_{sv}$ to ensure the convergence,
and find the $n_{sv}$ desired singular triplets.

\subsection{Semi-definiteness of the approximate spectral projector and its accuracy}

We have proved the eigenvalues
$\gamma_i\in [0,1]$ of the approximate spectral projector
$P$ in \Cref{sec: cross product method}.
We now confirm this property numerically and get more insights into
sizes of the $\gamma_i$.

For GL7d12, by \cite{davis2011university}, it is known that
the right-most and left-most singular values in the
interval $[11,12]$ are the 18-th largest one
and the 34-th largest one, respectively.
For flower\_5\_4, the right-most and left-most singular values
in the interval $[4.1,4.3]$ are the 214-th largest one and the 350-th
largest one, respectively.
The ends of these two intervals are not singular values of the matrices,
and the eigenvalues of the spectral projector $P_{S}$ are thus 1 and 0.
We choose $d$ in \eqref{dchoice} using $D=2$ and $4$,
compute the eigenvalues $\gamma_i,\ i=1,2,\ldots,n$ of $P$, and
depict the eigenvalues $\gamma_i$ of $P$ corresponding to
$\sigma_i\in [a,b]$ and some neighbors outside in \Cref{fig:eigenvalues}.
We record the key quantities $\|P_{S}-P\|$, $\gamma_{n_{sv}}$ and $\gamma_{n_{sv}+1}$,
the largest $\gamma_1$ and smallest $\gamma_n$, and
$\gamma_{p+1}$ for $p=\lceil \mu H_{M}\rceil $
by taking $\mu=1.1, 1.3, 1.5$, respectively, and list them
in \Cref{table:eigenvalues}.

\begin{figure}[tbhp]
    \centering
    \subfloat[GL7d12, $D=2$]{\includegraphics[scale=0.4]{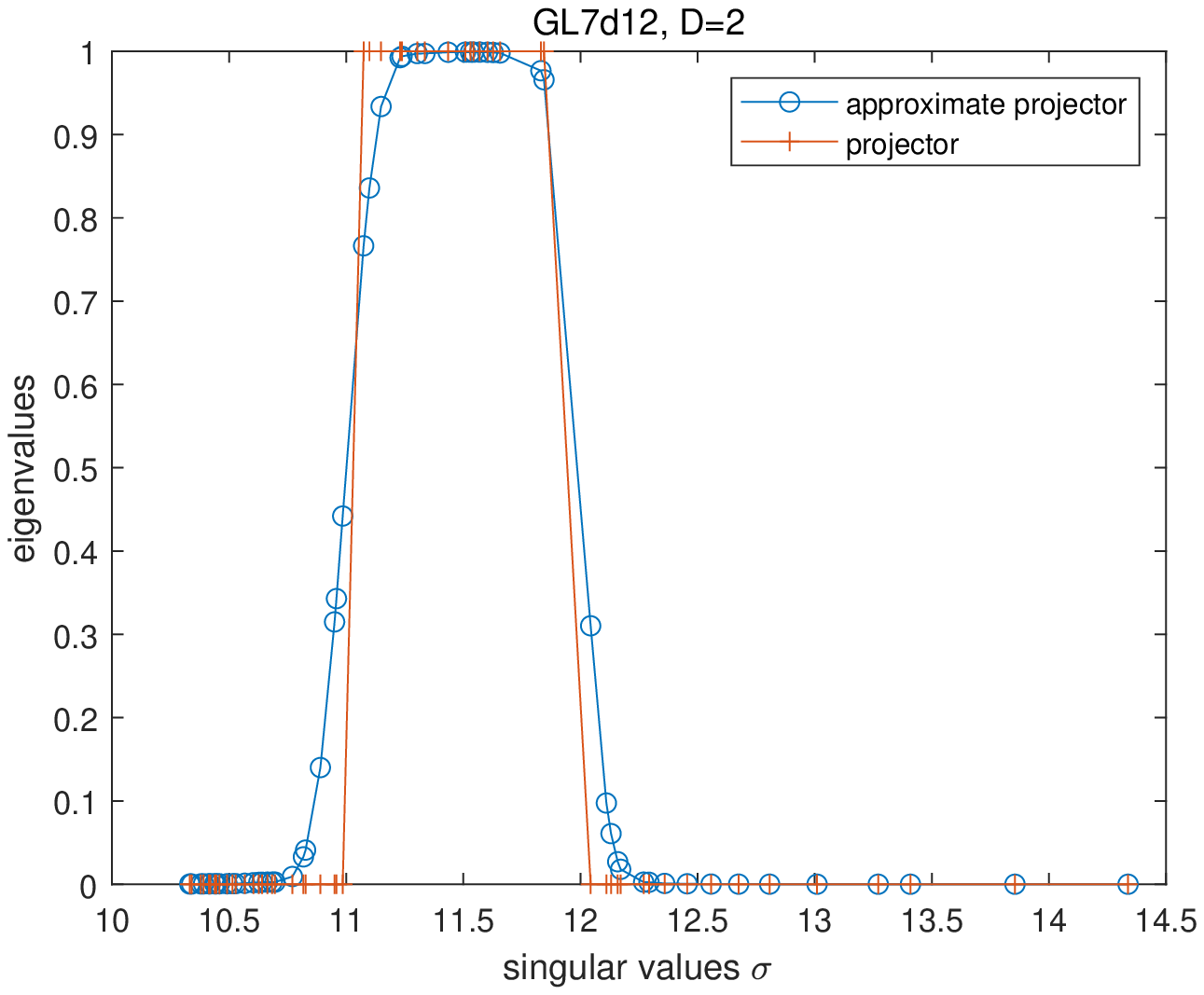}}
    \subfloat[GL7d12, $D=4$]{\includegraphics[scale=0.4]{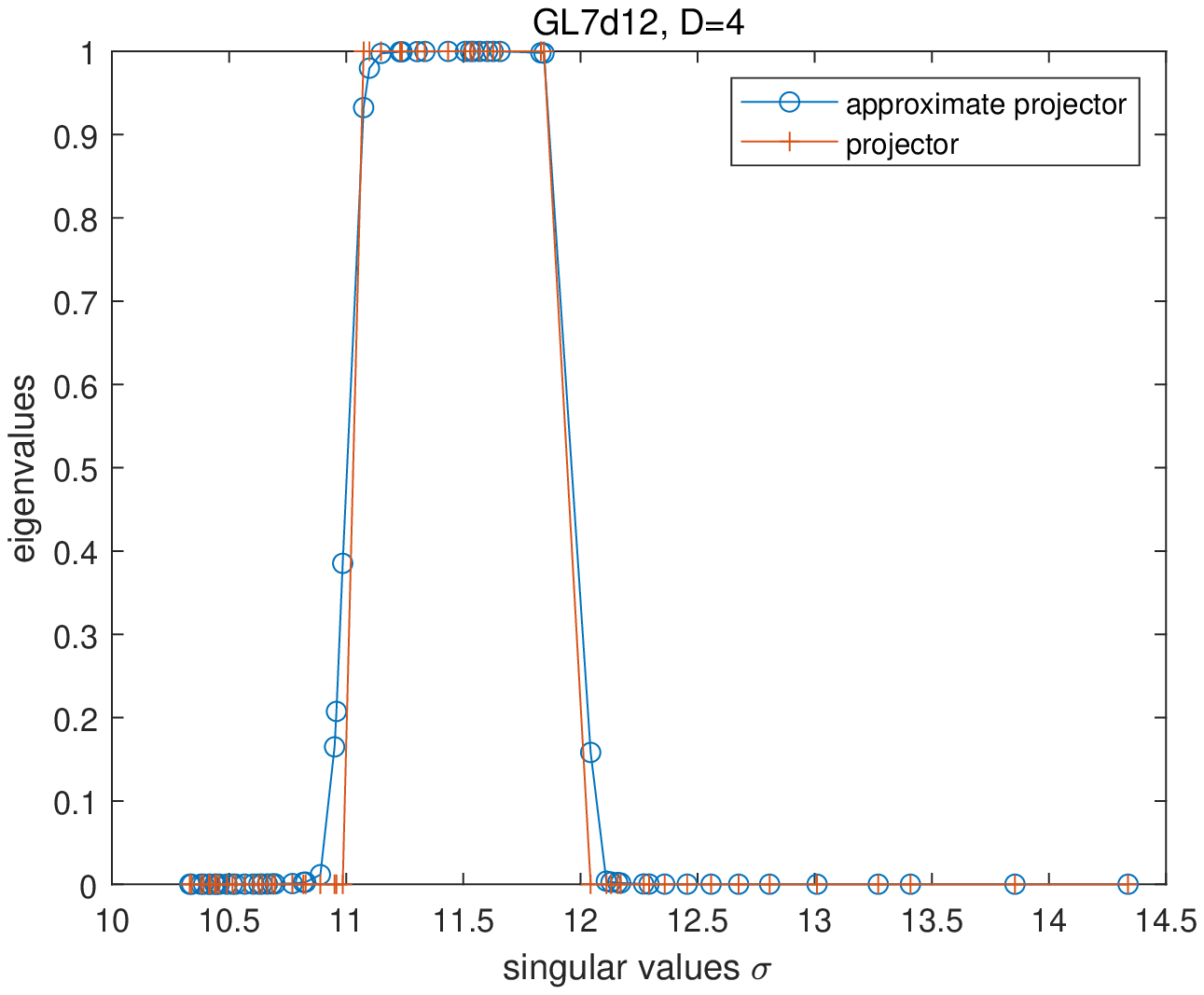}}

    \subfloat[flower\_5\_4, $D=2$]{\includegraphics[scale=0.4]{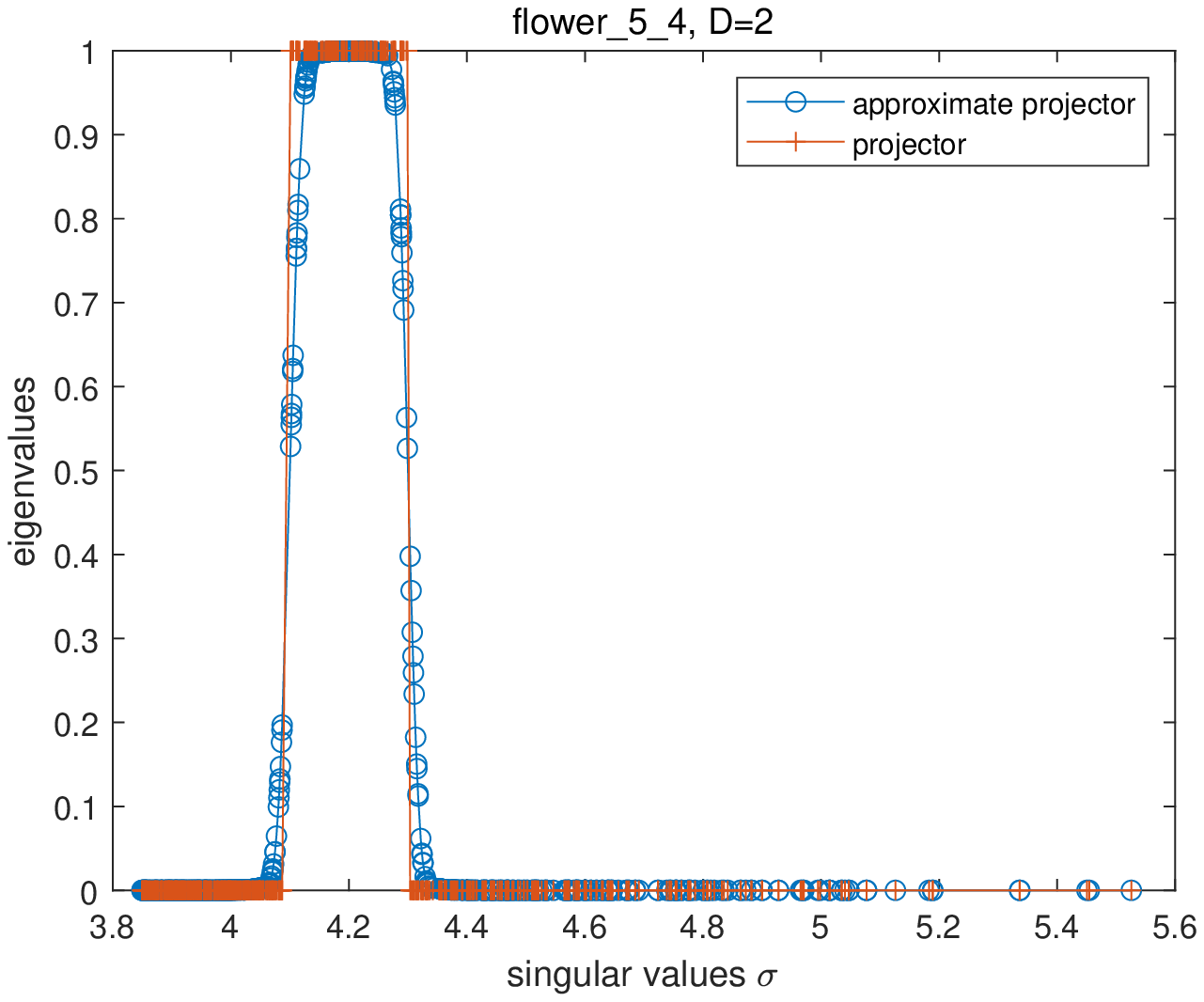}}
    \subfloat[flower\_5\_4, $D=4$]{\includegraphics[scale=0.4]{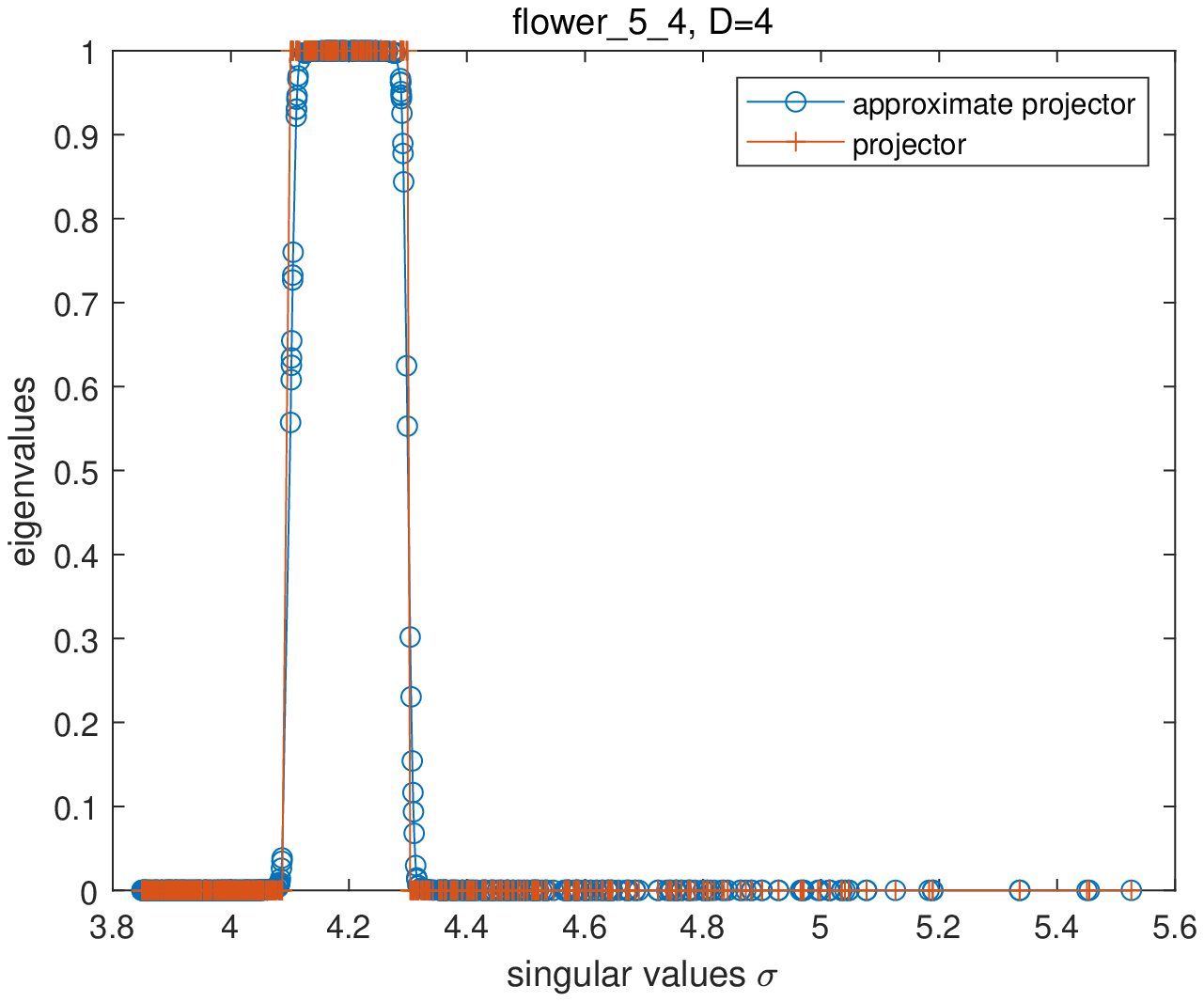}}
    \caption{The partial eigenvalues of $P$.}
    \label{fig:eigenvalues}
\end{figure}

\begin{table}[htbp]
    \caption{$\|P_S-P\|$ and some eigenvalues of $P$. The series degree $d$ for
    GL7d12 are $137$ and $276$, and the series degree $d$ for flower\_5\_4 are 365 and 732.}
    \label{table:eigenvalues}
    \centering
    \resizebox*{\textwidth}{!}
    {
    \begin{tabular}{|c|c|c|c|c|c|c|c|c|c|}
    \hline
    \multirow{2}{*}{Matrix} & \multirow{2}{*}{$D$} & \multirow{2}{*}{$\|P_{S}-P\|$} & \multirow{2}{*}{$\gamma_1$} & \multirow{2}{*}{$\gamma_{n_{sv}}$}& \multirow{2}{*}{$\gamma_{n_{sv}+1}$}  & \multicolumn{3}{c|}{$\gamma_{p+1}$} & \multirow{2}{*}{$\gamma_n$}                                                 \\
    \cline{7-9}
                                &            & & & &      & $\mu=1.1$  & $\mu=1.3$  & $\mu=1.5$ & \\
    \hline
    \multirow{2}{*}{GL7d12}    &   $2$    &0.4420 &  0.9990 & 0.7664 & 0.4420  & $3.15e-1$  & $9.75e-2$  & $2.70e-2$ & $1.63e-7$\\
                                &   $4$   &0.3852 &  0.9999 & 0.9323 & 0.3852   & $1.65e-2$  & $3.71e-3$  & $2.02e-3$ & $2.13e-8$\\
    \hline
    \multirow{2}{*}{flower\_5\_4}   &   $2$    & 0.4736 & 0.9996 & 0.5264 & 0.3978   & $1.45e-1$  & $7.92e-3$  & $1.95e-3$ & $5.28e-9$ \\
                                    &   $4$    & 0.4472 & 0.9999 & 0.5528 & 0.3017   & $1.30e-2$  & $6.84e-4$  & $1.93e-4$ & $6.72e-10$ \\
    \hline
    \end{tabular}
    }

\end{table}

Several comments are in order on the figure and the table.
First, for each matrix,
the two $P$ generated by the two $D$ are all SPSD since all the $\gamma_n>0$.
Second, the eigenvalues of each $P$ are indeed in $[0,1]$ since all the
$\gamma_1<1$ and are close to one; $\|P_S-P\|<\frac{1}{2}$, and
$\|P_S-P\|=1-\gamma_{n_{sv}}$ or $\gamma_{n_{sv}+1}$. Third, the
$\gamma_i$ decay to zero fast outside the given interval, and their
sizes indeed differ greatly as $i$ increases,
which justifies \Cref{rem:dchoice for nsv}. Fourth, the bigger $D$ is,
the larger $\gamma_{n_{sv}}$ but the smaller $\gamma_{n_{sv}+1}$ and $\gamma_p$
are, meaning that
the algorithm converges faster as $D$, i.e., the series degree $d$, increases.
Observe from \eqref{dchoice} that $d+2$ is exactly a multiple of $D$.
Insightfully, by a careful comparison,
we have found that, for $D=4$, the corresponding
$\gamma_{p+1}$ and $\gamma_n$ are approximately reduced by eight times, compared
to those for $D=2$. They indicate that these quantities are approximately
proportional to  $1/(d+2)^3$, and thus numerically justified \Cref{rem:error}.
Fifth, for each $D$, the slowest convergence factor $\frac{\gamma_{p+1}}{\gamma_{n_{sv}}}
<\frac{\gamma_{n_{sv}+1}}{\gamma_{n_{sv}}}$ considerably
as $\mu$, i.e., $p$, increases,
which shows that increasing $p$ can speed up the convergence of the CJ-FEAST
SVDsolver considerably. Sixth, all the
$\frac{\gamma_{p+1}}{\gamma_{n_{sv}}}<1$ considerably for the given
$\mu\in [1.1, 1.5]$ and $D=2, 4$, which implies that the algorithm converges quite quickly.
Visually, we plot the seven eigenvalues in
\Cref{table:eigenvalues} as \Cref{fig: some eigenvalues of P}, and show
how different they are for the two $D$. As is seen, the three $\gamma_{p+1}$
and $\gamma_n$ are reduced roughly one order from $D=2$ to $D=4$.

\begin{figure}[tbhp]
    \centering
    \subfloat[GL7d12, some eigenvalues of $P$]{\includegraphics[scale=0.4]{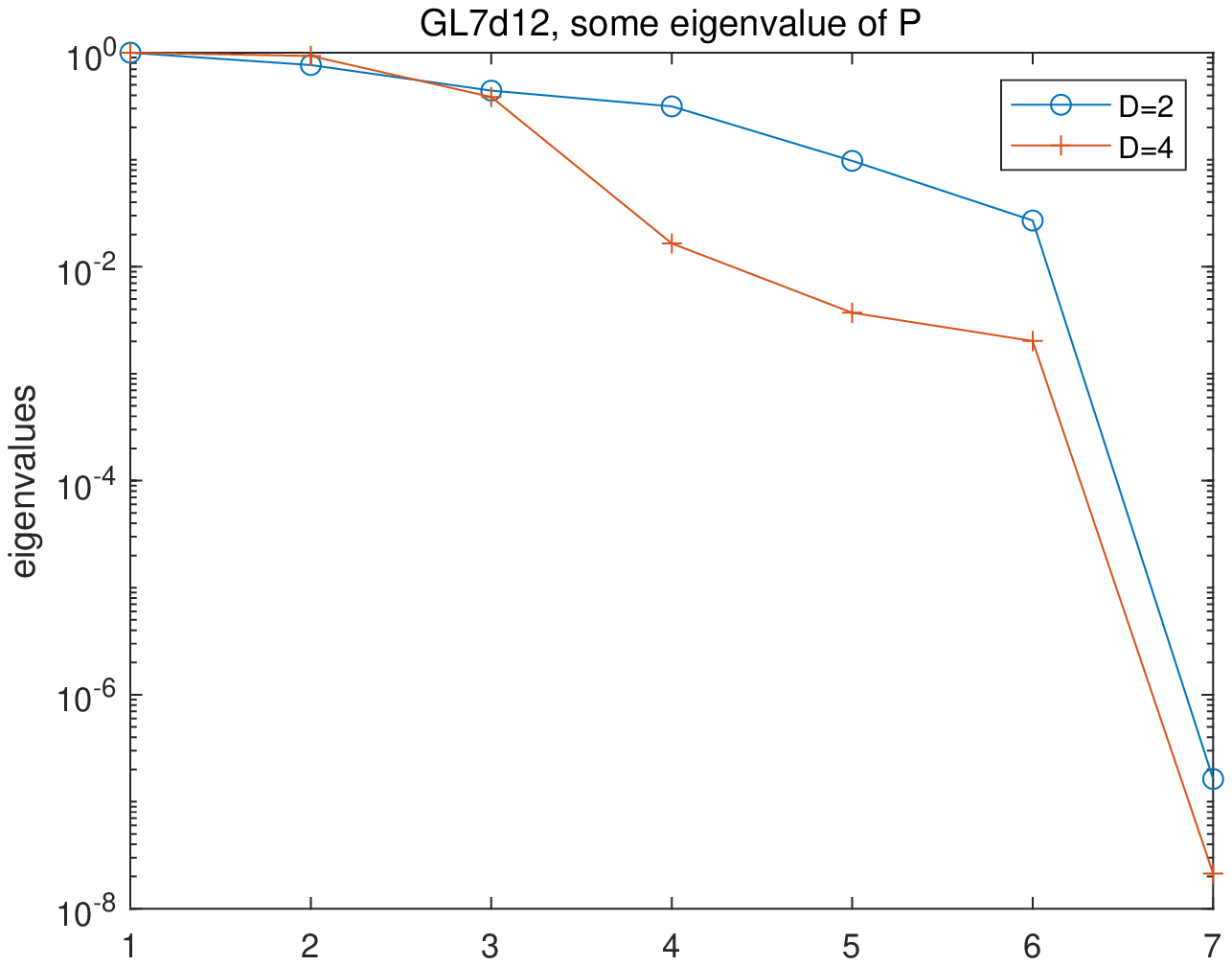}}
    \subfloat[flower\_5\_4, seven eigenvalue of $P$]{\includegraphics[scale=0.4]{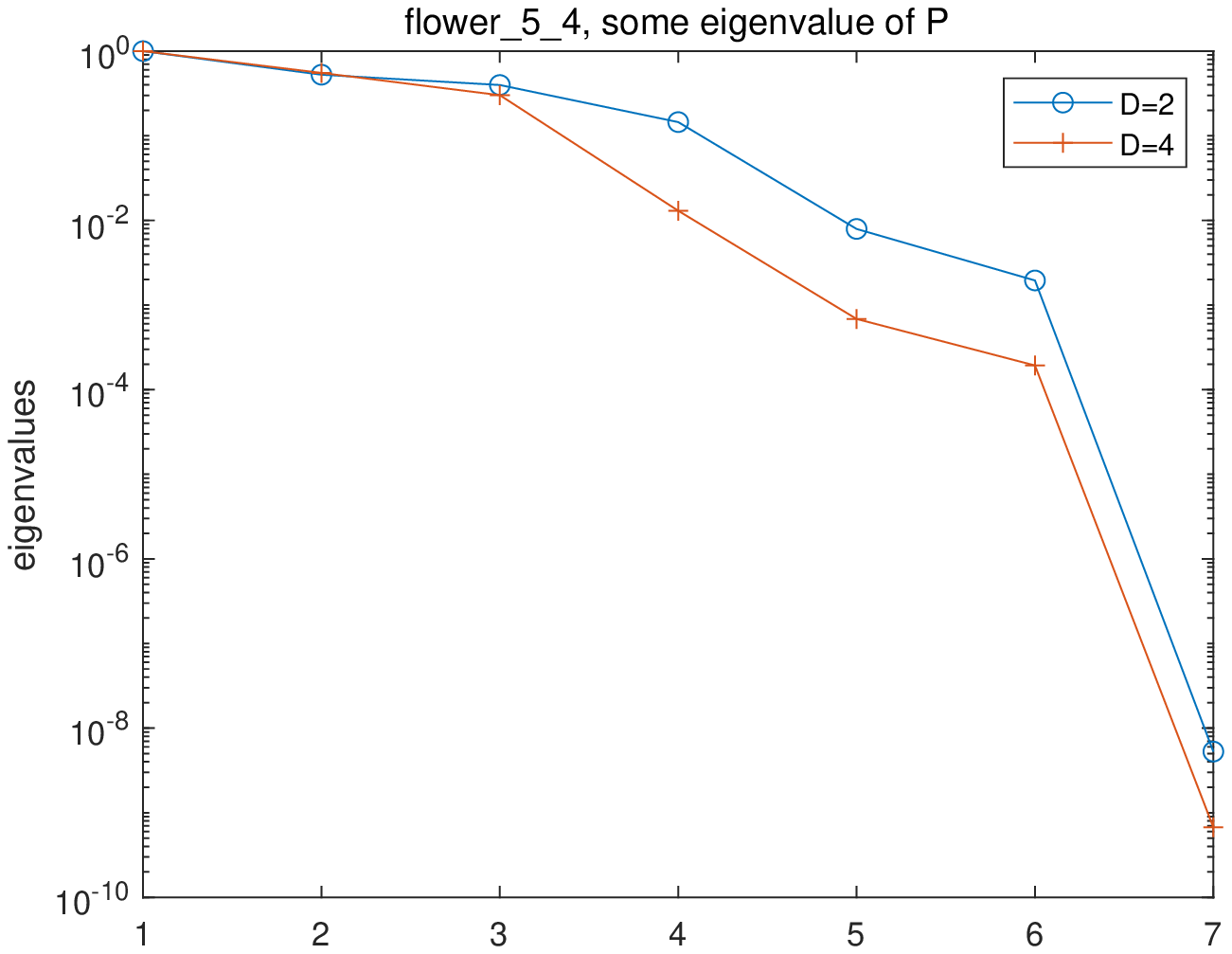}}
    \caption{Seven eigenvalues of $P$}
    \label{fig: some eigenvalues of P}
\end{figure}

The above results and analysis indicate that $\mu\in [1.1,1.5]$ for
a small $D$ are practical and work well.

\subsection{A comparison of CJ-FEAST SVDsolver and IFEAST}
In this subsection we numerically
compare \Cref{alg:Crossproduct-PSVD} with the contour integral-based
FEAST algorithm with inexact linear system solves, abbreviated as IFEAST
\cite{gavin2018ifeast,polizzi2020feast}, which can be directly
adapted to the SVD problem under consideration.
We use IFEAST to construct $P$ and then use \Cref{alg:subspace iteration}.
The unique fundamental difference between \Cref{alg:Crossproduct-PSVD} and
IFEAST is the construction way of $P$, and
all the other steps are the same.

IFEAST can use some flexible parameters  \cite[Section 3.1]{polizzi2020feast},
such as different contours and numerous numerical quadrature rules.
Here, for a given interval $[a,b]$ of interest,
we use the circle with the center $\frac{a^2+b^2}{2}$ and radius  $\frac{b^2-a^2}{2}$ as the contour.
We use the trapezoidal rule with eight nodes
and the Gauss--Legendre quadrature with sixteen nodes on
the circle, respectively, which are default parameters
in \cite[Section 3.1]{polizzi2020feast} and are also used in
\cite{tang2014feast} and \cite{guttel2015zolotarev}.
At each iteration, the resulting shifted linear systems are solved by BiCGstab,
as suggested in \cite{polizzi2020feast},
with increasing accuracy and the parameter $\alpha = 0.01$ \cite[Section 2]{gavin2018ifeast},
and the maximum iteration number is set to $n$, i.e., the problem size of
shifted linear systems.

Notice that $A^TA$ is real symmetric and the quadrature nodes are
symmetric with respect to the real axis, IFEAST only
needs to solve $p\times \frac{nnode}{2}$ linear
systems at each subspace iteration step,
where $nnode$ is the number of nodes on the circle. As we have
addressed, just as those shifted linear systems can be solved in parallel,
Step 3 of the CJ-FEAST SVDsolver
can be implemented in parallel too; that is, each of the $p$ matrix-vector
products is computed in a separate processor. More
precisely, one may solve the shifted linear systems in
$p\times \frac{nnode}{2}$ processors, while the $p$ MVs in CJ--FEAST SVDsolver
can be performed only in $p$ processors at each iteration.  As a result,
for a fair comparison, we only record the sequential MVs \cite{gavin2018ifeast},
which is the sum of the most MVs that BiCGstab uses for the one
among these linear systems at each subspace iteration step.
Notice that, for the shifted linear systems resulting from the matrices $A^TA$,
one iteration of BiCGstab costs four MVs with $A$ and $A^T$.
Keep in mind that if Step 3 of \Cref{alg:Crossproduct-PSVD} is
implemented in parallel then it consumes $2d$ sequential MVs for one subspace iteration step.
It is fair to use the sequential MVs to measure the overall efficiency of CJ-FEAST and IFEAST.

We take the same initial $\mathcal{V}^{(0)}$ and the same subspace
dimension $p$ as \eqref{pchoice} with $\mu = 1.2, 1.5$,
and $H_M$ is the closest one to $n_{sv}$ selected from \Cref{tab:number estimation}.
For \Cref{alg:Crossproduct-PSVD}, we choose the series degree $d$
using \eqref{dchoice} with $D=1, 2, 3$, respectively.
We record the sequential MVs and the number of iterations $k$ that
the norms of all the desired approximate singular triplets drop below a
prescribed tolerance $tol$, and denote them by SeqMVs($k$).
Moreover, we use the speedup ratio (SR) to compare the efficiency,
where SR is equal to the ratio of the SeqMVs of IFEAST over
the mean value of the three SeqMVs of \Cref{alg:Crossproduct-PSVD}.
Therefore, SR is the efficiency multiple of CJ-FEAST over IFEAST,
and SR $>1$ indicates that \Cref{alg:Crossproduct-PSVD} is more efficient;
otherwise, \Cref{alg:Crossproduct-PSVD} is less efficient.
\Cref{tab:seqMVs1e-8} and \Cref{tab:seqMVs1e-12} list the results obtained
for $tol=1e-8$ and $tol=1e-12$, respectively, where we have abbreviated
the trapezoidal rule and the Gauss--Legendre quadrature as T and G, respectively.

    \begin{table}[htbp]
        \caption{Computational results of IFEAST and \Cref{alg:Crossproduct-PSVD} with $tol=1e-8$.}
        \label{tab:seqMVs1e-8}
        \centering
        \resizebox*{\textwidth}{!}{
        \begin{tabular}{|c|c|c|c|c|c|c|c|c|}

        \hline
        \multirow{3}{*}{Matrix} & \multirow{3}{*}{$p$} & \multicolumn{5}{c|}{SeqMVs$(k)$}  & \multicolumn{2}{c|}{SR}       \\ \cline{3-9}
                          & &\multicolumn{2}{c|}{IFEAST} & \multicolumn{3}{c|}{\Cref{alg:Crossproduct-PSVD}}  & \multirow{2}{*}{T} & \multirow{2}{*}{G}      \\ \cline{3-7}
                          &  & T   &  G         & $D=1$ & $D=2$ & $D=3$  &  &   \\ \hline
        \multirow{2}{*}{GL7d12}   & 21          & 2906(11)  &  4418(4)    & 2448(18) & 2466(9) & 2484(6) & 1.2 & 1.8 \\
        &26           & 1962(8)  &  3630(4)              & 1904(14) & 1370(5) & 1656(4) & 1.2 & 2.2 \\ \hline
        \multirow{2}{*}{plat1919} & 10    &  410(7)      & 588(4)   & 756(14) & 672(6) & 680(4) & 0.6 & 0.8 \\
        &    12          &  398(7)  &   450(3)     & 756(14) & 672(6) & 510(3)  & 0.6 & 0.7 \\ \hline
        \multirow{2}{*}{flower\_5\_4} & 163   & 9794(13)    & 16320(4)   & 5824(16) & 4380(6) & 4392(4) & 2.0 & 3.3\\
        & 204  &    8202(7)       &  23336(4)   & 2548(7) & 2920(4) & 3294(3) & 2.8 & 8.0 \\ \hline
        \multirow{2}{*}{fv1} &  108   & 31176(8)        & 82318(4)  & 12120(6) & 12132(3) & 18198(3) & 2.2 & 5.8 \\
        &     135         &    21954(6) &  58902(3)             & 8080(4) & 12132(3) & 12132(2) & 2.0 & 5.5 \\ \hline
        \multirow{2}{*}{3elt\_dual}  & 443   &  20670(16)        & 37450(4)   & 10620(18) & 9472(8) & 8890(5) & 2.1 & 3.6 \\
        & 553             & 9786(7)  & 30962(4)              & 4130(7) & 4736(4) & 5334(3) & 2.1 & 6.5 \\ \hline
        \multirow{2}{*}{rel8} & 16         & 3910(16) & 3592(4)   & 6160(28) & 4884(11) & 4008(6) & 0.8 & 0.7 \\
        &  20        & 2458(10) &  2612(3)       & 3300(15) & 2220(5) & 2672(4) & 0.9 & 1.0 \\ \hline
        \multirow{2}{*}{crack\_dual}  & 397   &  84646(25)        &  61700(5) & 19314(29) & 17368(13) & 16032(8) & 4.8 & 3.5\\
        & 496            & 21170(13)  & 52332(5)        & 11332(17) & 8016(6) & 8016(4) & 2.3 & 5.7 \\ \hline
        \multirow{2}{*}{nopoly}  & 406   & 51228(15)         & 116320(5)  & 19008(18) & 14798(7) & 12696(4) & 3.3 & 7.5 \\
        &  507           & 20912(7)  & 66756(4)    & 5280(5) & 8456(4) & 9522(3) & 2.7 & 8.6 \\ \hline
        \multirow{2}{*}{barth5}  &  460  &  101058(23)   & 148510(5) & 18828(18) & 12564(6) & 12576(4)  & 6.9 & 10.1\\
        & 574            & 31912(8)  & 88370(4)    & 5230(5) & 8376(4) & 9432(3) & 4.2 & 11.5\\ \hline
        \multirow{2}{*}{L-9}  & 576   & 68666(16)  & 155076(5) & 14970(15) & 11988(6) & 11992(4) & 5.3 & 11.9 \\
        & 720            & 40714(9) & 91302(4)    & 4990(5) & 7992(4) & 8994(3) & 5.6 & 12.5 \\ \hline
        \end{tabular}
            }
    \end{table}

    \begin{table}[htbp]
        \caption{Computational results of IFEAST and \Cref{alg:Crossproduct-PSVD} with $tol=1e-12$.}
        \label{tab:seqMVs1e-12}
        \centering
        \resizebox*{\textwidth}{!}{
        \begin{tabular}{|c|c|c|c|c|c|c|c|c|}
        \hline
        \multirow{3}{*}{Matrix} & \multirow{3}{*}{$p$} & \multicolumn{5}{c|}{SeqMVs$(k)$}   & \multicolumn{2}{c|}{SR}        \\ \cline{3-9}
                          & &\multicolumn{2}{c|}{IFEAST} & \multicolumn{3}{c|}{\Cref{alg:Crossproduct-PSVD}}  & \multirow{2}{*}{T} & \multirow{2}{*}{G}     \\ \cline{3-7}
                          &  & T   &  G         &$ D=1$ & $D=2$ & $D=3$ & & \\ \hline
        \multirow{2}{*}{GL7d12}   & 21     & 6576(17)  &  7052(5)   & 4352(32) & 4110(15) & 3726(9)  & 1.6 & 1.7 \\
        &  26  &  5312(14)     & 7598(7)             & 2992(22) & 2466(9) & 2070(5)  & 2.1 & 3.0 \\ \hline
        \multirow{2}{*}{plat1919} &  10            & 666(12)   &  674(5)   & 1188(22) & 1008(9) & 850(5)  & 0.7 & 0.7 \\
        & 12    &  522(10)      & 714(5)                  & 1080(20) & 896(8) & 850(5)  & 0.6 & 0.8 \\ \hline
        \multirow{2}{*}{flower\_5\_4} & 163  &  57162(20)         & 35764(5)  & 9464(26) & 7300(10) & 5490(5)  & 7.7 & 4.8 \\
        &  204  &  14568(9)        & 26788(5)  & 3276(9) & 3650(5) & 5496(5)  & 3.5 & 6.5 \\ \hline
        \multirow{2}{*}{fv1} & 108             & 193024(14)    & 195936(7)              & 16160(8) & 20220(5) & 24264(4)  & 9.5 & 9.7 \\
        & 135    & 76036(7)        &   95756(4) & 10100(5) & 16176(4) & 18198(3)  & 5.1 & 6.5 \\ \hline
        \multirow{2}{*}{3elt\_dual} &  443            & 174642(24)   & 109562(6)    & 18880(32) & 13024(11) & 10668(6)  & 12.3 & 7.7 \\
        &  553  &  47492(11)        &   102634(6)   & 5310(9) & 5920(5) & 8890(5)  & 7.1 & 15.3 \\ \hline
        \multirow{2}{*}{rel8} &  16        & 7998(22) & 5818(5)        & 8140(37) & 6660(15) & 5344(8)  & 1.2 & 0.9\\
                            & 20  &  5304(14)     &  7630(5)         & 4620(21) & 3552(8) & 3340(5)   & 1.4 & 2.0 \\ \hline
        \multirow{2}{*}{crack\_dual} & 397            & 371568(38)  & 225696(8) & 31968(48) & 25384(19) & 22044(11)  & 14.0 & 8.5 \\
        &  496  &  192250(21)        & 178956(7)      & 13986(21) & 10688(8) & 12024(6)  & 15.7 & 14.6 \\ \hline
        \multirow{2}{*}{nopoly}  & 406   & 314438(26)    & 125280(6)  & 26400(25) & 19026(9) & 19044(6)  & 14.6 & 5.8 \\
        &  507           & 82246(10)  & 75062(5)        & 6336(6) & 10570(5) & 12696(4)  & 8.3 & 7.7 \\ \hline
        \multirow{2}{*}{barth5}  &  460  & 889820(39) & 175112(6) & 27196(26) & 20940(10) & 18864(6)  & 39.8 & 7.8 \\
        & 574            & 150860(10)  & 113838(5)    & 8368(8) & 10470(5) & 12576(4)  & 14.4 & 10.9 \\ \hline
        \multirow{2}{*}{L-9}  & 576   & 471596(25) & 193192(6) & 23952(24) & 19980(10) & 17988(6)  & 22.8 & 9.4 \\
        & 720            & 220778(12) & 145958(6)    & 6986(7) & 9990(5) & 11992(4)  & 22.9 & 15.1 \\ \hline
        \end{tabular}
        }
    \end{table}

Let us analyze \Cref{tab:seqMVs1e-8} and \Cref{tab:seqMVs1e-12}.
For GL7d12, plat1919 and rel8, since
the intervals of interest are close to the right end of the singular spectra,
the shifted linear systems involved in IFEAST are not very
indefinite by noticing
that most of the eigenvalues of the coefficient matrices are in the left
half plane and only a handful of them are in the right
half plane. It is known that, for such linear systems,
Krylov iterative solvers such as BiCGstab and GMRES may converge relatively faster.
As the SR columns of tables indicate, the SeqMVs consumed
by IFEAST and CJ-FEAST are comparable, meaning that the two
algorithms are almost equally efficient and there is no obvious winner.
But for the other seven test problems,
the intervals of interest are truly inside the singular spectra, that is,
the desired singular values are some relatively interior ones,
so that the linear systems in the IFEAST may be highly indefinite,
which make Krylov iterative solvers possibly converge very slowly.
For these SVD problems, we see from the SR columns of tables that
CJ-FEAST is a few and often tens times more
efficient than IFEAST, very substantial improvements.

We have more findings.
When the stopping tolerance
$tol$ changes from $1e-8$ to $1e-12$, although
the outer iterations needed increase regularly for each problem and given
parameters, the SeqMVs($k$) consumed by IFEAST may increase dramatically, which are especially
true when the intervals of interest are inside the singular spectra.
By inspecting the convergence processes of BiCGstab for solving shifted
linear systems at each outer iteration, we have observed that it
became much harder for BiCGstab to reduce the residual norms of shifted systems
as outer iterations proceed and approximate
singular triplets converge. In fact, we have found that once outer residual norms
are around $1e-11$, BiCGstab often consumed considerably many iterations to meet
the desired stopping criterion in subsequent outer iterations.
In contrast, CJ-FEAST always converges linearly and regularly, and
the SeqMVs used by it thus increase regularly from $tol=1e-8$ to $tol=1e-12$. This can be seen
from the SR columns of the tables, where CJ-FEAST is more advantageous to
the two contour integral-based IFEAST solvers for $tol=1e-12$.

Finally, we test the CJ-FEAST SVDsolver and IFEAST on the problem
shuttle\_eddy with $tol=1e-13$, which, though smaller, is considerably
bigger than $\mathcal{O}(\epsilon_{\rm mach})$. We take
$p = \lceil 1.2H_M\rceil = \lceil 1.2 \times 6.1\rceil = 8$,
where $H_M = 6.1$ is the closest to $n_{sv}=6$ selected
from \Cref{tab:number estimation}.
For IFEAST, we plot the convergence processes of the biggest relative
residual norms among the six ones in \Cref{fig: shuttle convergence process}.
For CJ--FEAST with $D = 1$ and 2, which corresponds to $d=64914$ and $129830$,
the convergence processes of the six Ritz approximations are similar,
and we plot the residual norms of one Ritz approximation with $D=1$ and $D=2$
in \Cref{fig: shuttle convergence process}, respectively.
We also take a closer look at the convergence behavior  of IFEAST
and plot \Cref{fig: shuttle closer look} after the residual norms drop below
$1e-11$, which exhibits the subsequent convergence
process more clearly and visually.

Several comments are made.
First, it is observed from \Cref{fig: shuttle convergence process} that both
IFEAST and CJ--FEAST converge quite fast until the residual norm decreases to
$1e-11$. After that,
IFEAST with the trapezoidal rule starts to stabilize
above $tol=1e-13$ in subsequent iterations
and IFEAST with the Gauss--Legendre quadrature succeeds but
converges irregularly,
while CJ--FEAST with $D=1,2$ performs regularly and
the residual norms drops below $tol=1e-13$ at iterations
$k=3$ and $2$, respectively.
Second, if $tol=1e-11$ then all the residual norms
of six desired triplets computed by IFEAST with the trapezoidal rule and
Gauss--Legendre quadrature drop below $1e-11$ at iterations $k=9,8$,
respectively, and the SeqMVs are 361338 and 333572;
the SeqMVs$(k)$ consumed by CJ--FEAST with $D=1$ and 2
are $389484(3)$ and $519320(2)$, respectively. Therefore,
CJ-FEAST with $D=1$ is as efficient as IFEAST if $tol=1e-11$.
Third, \Cref{fig: shuttle closer look} shows that IFEAST with the
Gauss--Legendre quadrature behaves irregularly but the residual norm
ultimately drops below the prescribed $tol=1e-13$ at $k=31$,
while the residual norms obtained by IFEAST with the trapezoidal rule decrease
faster and more regularly but almost stagnate from
$k=15$ upwards with the residual norms bigger than $tol$.
Fourth, the residual norms computed by CJ--FEAST with $D = 1$ and 2
further decrease and achieve the prescribed tolerance very quickly.
As a matter of
fact, the residual norms of six Ritz approximations computed by CJ--FEAST
with $D=1$ are already $4.64e-15, 5.61e-15, 4.94e-15, 5.91e-15, 5.87e-15,
9.25e-15$ and with $D=2$ are $6.62e-15, 6.07e-15, 6.00e-15, 6.59e-15, 6.25e-15,
6.36e-15$, respectively. All of them are $\mathcal{O}(\epsilon_{\rm mach})$.
Therefore, for this problem, \Cref{alg:Crossproduct-PSVD} is more robust
than IFEAST when higher accuracy is required. More generally, we have found
that CJ-FEAST works well
for a prescribed tolerance $tol=\mathcal{O}(\epsilon_{\rm mach})$, but IFEAST
may fail to converge for $tol=1e-13$ or smaller but no less than
$\mathcal{O}(\epsilon_{\rm mach})$ for some problems, due to
the solutions of shifted linear systems in finite precision.

\begin{figure}[tbhp]
    \centering
    \subfloat[convergence process]{\label{fig: shuttle convergence process}\includegraphics[scale=0.4]{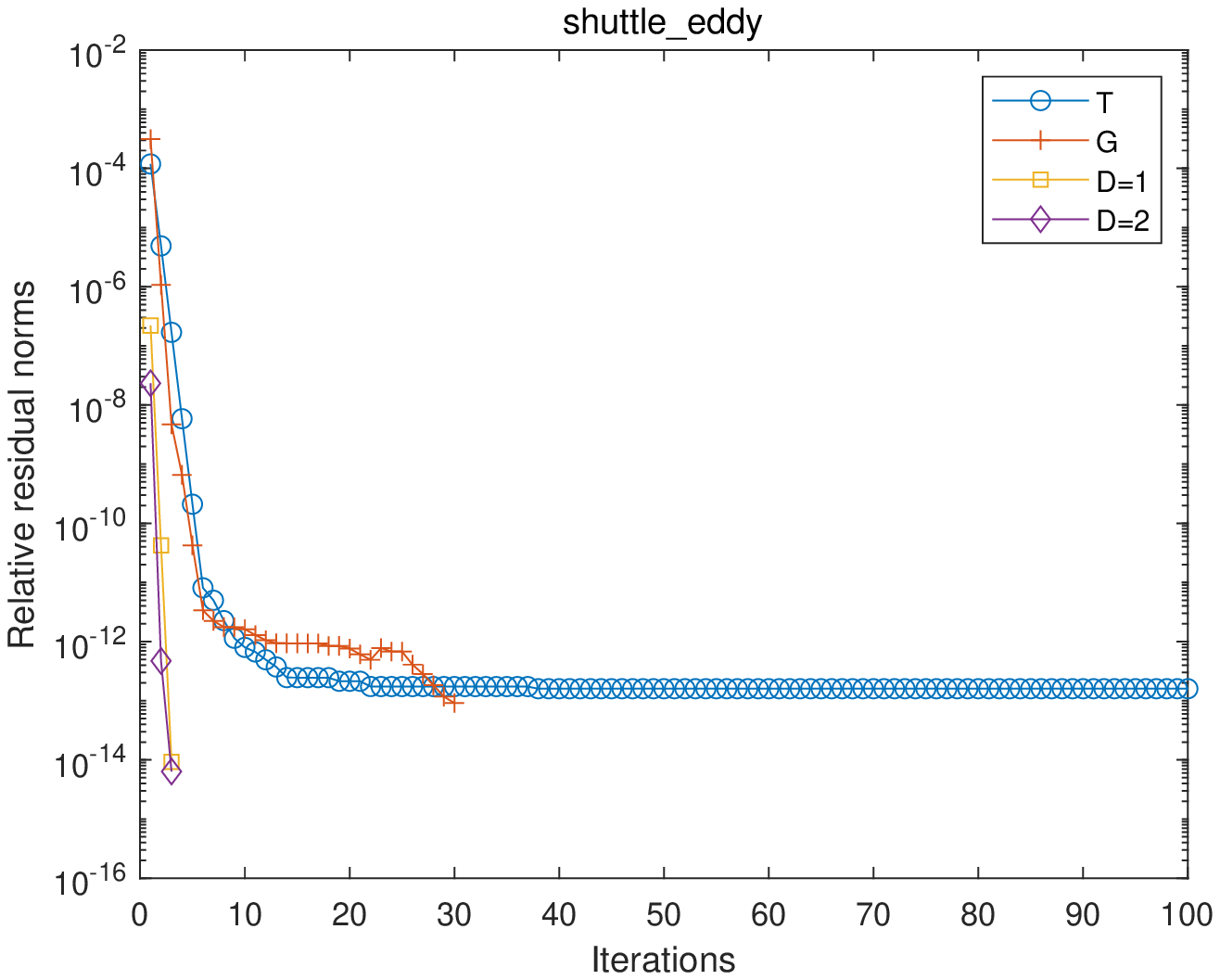}}
    \subfloat[A closer look]{\label{fig: shuttle closer look}\includegraphics[scale=0.4]{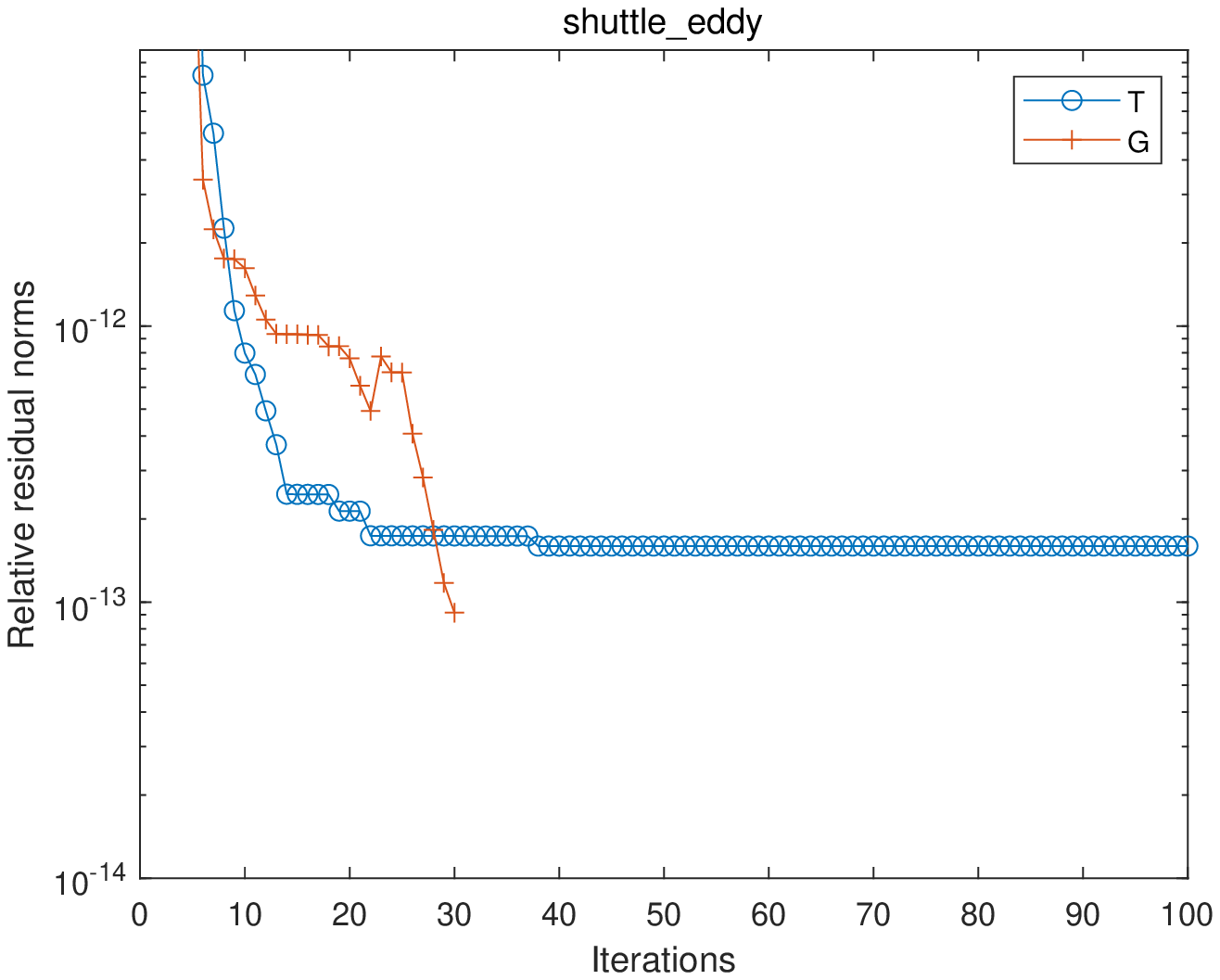}}
    \caption{Convergence process of shuttle\_eddy with $tol=1e-13$.}
    \label{fig:shuttle}
\end{figure}

\section{Conclusions}\label{sec:end}

We have considered the problem
of approximating the step function $h(x)$ in \eqref{hdef}
by the Chebyshev--Jackson polynomial series,
proved its pointwise convergence to $h(x)$, and derived quantitative pointwise
error bounds. Making use of these results,
we have established quantitative accuracy estimates for the
approximate spectral projector constructed by the series as an
approximation to the spectral projector $P_{S}$ of $A^TA$ associated with all
the singular values $\sigma\in [a,b]$. We have also proved that
the approximate spectral projector constructed by the Chebyshev--Jackson
series is unconditionally SPSD,
which enables us to reliably estimate the number $n_{sv}$ of desired singular
triplets and propose a robust selection strategy to ensure that the subspace
dimension $p\geq n_{sv}$. Based on these results, we have developed
the CJ-FEAST SVDsolver for the computation of the singular triplets of $A$
with $\sigma\in [a,b]$.
We have analyzed the convergence of the algorithm, and
proved how the subspaces constructed converge to the desired right
singular subspace and how each of the Ritz approximations converges
as iterations proceed.
In the meantime, we have discussed how to select
the subspace dimension $p$ and the series degree $d$
in computations, and proposed robust and general-purpose
selection strategies for them. We have numerically tested our
CJ-FEAST SVDsolver on a number of problems in several
aspects and shown that it is robust, effective and efficient.
Numerical experiments have demonstrated that
the CJ-FEAST SVDsolver is at least competitive with IFEAST and is
much more efficient than IFEAST when the desired singular values
are extreme and interior ones, respectively, and they have
also illustrated that CJ-FEAST is more robust than IFEAST if
a higher accuracy is required.

The adaptation of the CJ-FEAST SVDsolver to the real symmetric and complex
Hermitian eigenvalue problem is straightforward, where the eigenpairs
with the eigenvalues in a given real interval are of interest.
We only need to replace the Rayleigh--Ritz projection for
the SVD problem by the counterpart for the eigenvalue problem. The
results and analysis are directly applicable or adaptable to the variant of SVDsolver,
i.e., the CJ-FEAST eigensolver,
and the practical selection strategies
proposed for $p$ and $d$ still work.
Moreover, the construction of an approximate spectral projector by
the Chebyshev--Jackson series involves
only matrix-matrix products and can thus be implemented
very efficiently in parallel computing environments.
In a word, the CJ-FEAST eigensolver is an
efficient and robust alternative of
the available contour integral-based FEAST eigensolvers for
real symmetric or complex Hermitian eigenvalue problems.


%
\section*{Declarations}

{\bf Conflict of interest} \ The two authors declare that they have no
financial interests, and they read and approved the final manuscript.
The algorithmic Matlab code is available upon reasonable request from
the corresponding author.
\bigskip

\noindent {\bf Data Availability} \ Enquires about data availability should be 
directed to the authors.


\begin{thebibliography}{10}
\providecommand{\url}[1]{{#1}}
\providecommand{\urlprefix}{URL }
\expandafter\ifx\csname urlstyle\endcsname\relax
  \providecommand{\doi}[1]{DOI~\discretionary{}{}{}#1}\else
  \providecommand{\doi}{DOI~\discretionary{}{}{}\begingroup
  \urlstyle{rm}\Url}\fi

\bibitem{avron2011randomized}
Avron, H., Toledo, S.: Randomized algorithms for estimating the trace of an
  implicit symmetric positive semi-definite matrix.
\newblock J. ACM \textbf{58}(2), Art. 8, 17 (2011).
\newblock \doi{10.1145/1944345.1944349}

\bibitem{Cortinovis2021onrandom}
Cortinovis, A., Kressner, D.: On randomized trace estimates for indefinite
  matrices with an application to determinants.
\newblock Found. Comput. Math. \textbf{22}(3), 875--903 (2022).
\newblock \doi{10.1007/s10208-021-09525-9}

\bibitem{davis2011university}
Davis, T.A., Hu, Y.: The {U}niversity of {F}lorida sparse matrix collection.
\newblock ACM Trans. Math. Software \textbf{38}(1), Art. 1, 25 (2011).
\newblock \doi{10.1145/2049662.2049663}

\bibitem{di2016efficient}
Di~Napoli, E., Polizzi, E., Saad, Y.: Efficient estimation of eigenvalue counts
  in an interval.
\newblock Numer. Linear Algebra Appl. \textbf{23}(4), 674--692 (2016).
\newblock \doi{10.1002/nla.2048}

\bibitem{Futamura2014Online}
Futamura, Y., Sakurai, T.: z-{P}ares: Parallel {E}igenvalue {S}olver (2014).
\newblock \urlprefix\url{https://zpares.cs.tsukuba.ac.jp/}

\bibitem{gavin2018ifeast}
Gavin, B., Polizzi, E.: Krylov eigenvalue strategy using the {FEAST} algorithm
  with inexact system solves.
\newblock Numer. Linear Algebra Appl. \textbf{25}(5), e2188, 20 (2018).
\newblock \doi{10.1002/nla.2188}

\bibitem{golub2013matrix}
Golub, G.H., Van~Loan, C.F.: Matrix Computations, fourth edn.
\newblock Johns Hopkins Studies in the Mathematical Sciences. Johns Hopkins
  University Press, Baltimore, MD (2013)

\bibitem{guttel2015zolotarev}
G\"{u}ttel, S., Polizzi, E., Tang, P.T.P., Viaud, G.: Zolotarev quadrature
  rules and load balancing for the {FEAST} eigensolver.
\newblock SIAM J. Sci. Comput. \textbf{37}(4), A2100--A2122 (2015).
\newblock \doi{10.1137/140980090}

\bibitem{ikegami2010contour}
Ikegami, T., Sakurai, T.: Contour integral eigensolver for non-{H}ermitian
  systems: a {R}ayleigh-{R}itz-type approach.
\newblock Taiwanese J. Math. \textbf{14}(3A), 825--837 (2010).
\newblock \doi{10.11650/twjm/1500405869}

\bibitem{ikegami2010filter}
Ikegami, T., Sakurai, T., Nagashima, U.: A filter diagonalization for
  generalized eigenvalue problems based on the {S}akurai-{S}ugiura projection
  method.
\newblock J. Comput. Appl. Math. \textbf{233}(8), 1927--1936 (2010).
\newblock \doi{10.1016/j.cam.2009.09.029}

\bibitem{imakura2014SSarnoldi}
Imakura, A., Du, L., Sakurai, T.: A block {A}rnoldi-type contour integral
  spectral projection method for solving generalized eigenvalue problems.
\newblock Appl. Math. Lett. \textbf{32}, 22--27 (2014).
\newblock \doi{10.1016/j.aml.2014.02.007}

\bibitem{sakurai2016}
Imakura, A., Du, L., Sakurai, T.: Relationships among contour integral-based
  methods for solving generalized eigenvalue problems.
\newblock Jpn. J. Ind. Appl. Math. \textbf{33}(3), 721--750 (2016).
\newblock \doi{10.1007/s13160-016-0224-x}

\bibitem{jay1999electronic}
Jay, L.O., Kim, H., Saad, Y., Chelikowsky, J.R.: Electronic structure
  calculations for plane-wave codes without diagonalization.
\newblock Comput. Phys. Commun. \textbf{118}(1), 21--30 (1999).
\newblock \doi{10.1016/S0010-4655(98)00192-1}

\bibitem{jia1999}
Jia, Z.: Polynomial characterizations of the approximate eigenvectors by the
  refined {A}rnoldi method and an implicitly restarted refined {A}rnoldi
  algorithm.
\newblock Linear Algebra Appl. \textbf{287}(1-3), 191--214 (1999).
\newblock \doi{10.1016/S0024-3795(98)10197-0}

\bibitem{jia2003implicitly}
Jia, Z., Niu, D.: An implicitly restarted refined bidiagonalization {L}anczos
  method for computing a partial singular value decomposition.
\newblock SIAM J. Matrix Anal. Appl. \textbf{25}(1), 246--265 (2003).
\newblock \doi{10.1137/S0895479802404192}

\bibitem{jia2010refined}
Jia, Z., Niu, D.: A refined harmonic {L}anczos bidiagonalization method and an
  implicitly restarted algorithm for computing the smallest singular triplets
  of large matrices.
\newblock SIAM J. Sci. Comput. \textbf{32}(2), 714--744 (2010).
\newblock \doi{10.1137/080733383}

\bibitem{kestyn2016feast}
Kestyn, J., Polizzi, E., Tang, P.T.P.: F{EAST} eigensolver for non-{H}ermitian
  problems.
\newblock SIAM J. Sci. Comput. \textbf{38}(5), S772--S799 (2016).
\newblock \doi{10.1137/15M1026572}

\bibitem{lehoucq1998}
Lehoucq, R.B., Sorensen, D., Yang, C.: {ARPACK} {U}sers' {G}uide: {S}olution of
  {L}arge {S}cale {E}igenvalue {P}roblems by {I}mplicitly {R}estarted {A}rnoldi
  {M}ethods.
\newblock SIAM, Philadephia, PA (1998)

\bibitem{mason2002chebyshev}
Mason, J.C., Handscomb, D.C.: Chebyshev Polynomials.
\newblock Chapman \& Hall/CRC, Boca Raton, FL (2003)

\bibitem{parlett1998symmetric}
Parlett, B.N.: The {S}ymmetric {E}igenvalue {P}roblem, \emph{Classics in
  Applied Mathematics}, vol.~20.
\newblock SIAM, Philadelphia, PA (1998).
\newblock \doi{10.1137/1.9781611971163}

\bibitem{polizzi2009density}
Polizzi, E.: Density-matrix-based algorithm for solving eigenvalue problems.
\newblock Phys. Rev. B \textbf{79}(11), e115112, 6 (2009).
\newblock \doi{10.1103/PhysRevB.79.115112}

\bibitem{polizzi2020feast}
Polizzi, E.: {FEAST} eigenvalue solver v4.0 user guide (2020).
\newblock \doi{10.48550/arXiv.2002.04807}

\bibitem{rivlin1981introduction}
Rivlin, T.J.: An Introduction to the Approximation of Functions.
\newblock Dover Books on Advanced Mathematics. Dover Publications, Inc., New
  York (1981)

\bibitem{robbe2009}
Robb\'{e}, M., Sadkane, M., Spence, A.: Inexact inverse subspace iteration with
  preconditioning applied to non-{H}ermitian eigenvalue problems.
\newblock SIAM J. Matrix Anal. Appl. \textbf{31}(1), 92--113 (2009).
\newblock \doi{10.1137/060673795}

\bibitem{roosta2015improved}
Roosta-Khorasani, F., Ascher, U.: Improved bounds on sample size for implicit
  matrix trace estimators.
\newblock Found. Comput. Math. \textbf{15}(5), 1187--1212 (2015).
\newblock \doi{10.1007/s10208-014-9220-1}

\bibitem{saad2003}
Saad, Y.: Iterative Methods for Sparse Linear Systems, second edn.
\newblock SIAM, Philadelphia, PA (2003).
\newblock \doi{10.1137/1.9780898718003}

\bibitem{saad2011numerical}
Saad, Y.: Numerical Methods for Large Eigenvalue Problems, \emph{Classics in
  Applied Mathematics}, vol.~66.
\newblock SIAM, Philadelphia, PA (2011).
\newblock \doi{10.1137/1.9781611970739}

\bibitem{sakurai2003projection}
Sakurai, T., Sugiura, H.: A projection method for generalized eigenvalue
  problems using numerical integration.
\newblock J. Comput. Appl. Math. \textbf{159}(1), 119--128 (2003).
\newblock \doi{10.1016/S0377-0427(03)00565-X}

\bibitem{sakurai2007cirr}
Sakurai, T., Tadano, H.: C{IRR}: a {R}ayleigh-{R}itz type method with contour
  integral for generalized eigenvalue problems.
\newblock Hokkaido Math. J. \textbf{36}(4), 745--757 (2007).
\newblock \doi{10.14492/hokmj/1272848031}

\bibitem{sorensen1992}
Sorensen, D.C.: Implicit application of polynomial filters in a {$k$}-step
  {A}rnoldi method.
\newblock SIAM J. Matrix Anal. Appl. \textbf{13}(1), 357--385 (1992).
\newblock \doi{10.1137/0613025}

\bibitem{stewart2001matrix}
Stewart, G.W.: Matrix {A}lgorithms, {V}ol. {II}: Eigensystems.
\newblock SIAM, Philadelphia, PA (2001).
\newblock \doi{10.1137/1.9780898718058}

\bibitem{tang2014feast}
Tang, P.T.P., Polizzi, E.: F{EAST} as a subspace iteration eigensolver
  accelerated by approximate spectral projection.
\newblock SIAM J. Matrix Anal. Appl. \textbf{35}(2), 354--390 (2014).
\newblock \doi{10.1137/13090866X}

\end{thebibliography}

%
%

\end{document}